\documentclass[12pt]{article}%
\usepackage[hidelinks]{hyperref}
\usepackage{amssymb}
\usepackage{amsfonts}
\usepackage{amsmath}
\usepackage{amsthm}
\usepackage{psfrag, mathrsfs}
\usepackage{graphicx}%
\usepackage{overpic}
\providecommand{\U}[1]{\protect\rule{.1in}{.1in}}
\newtheorem{theorem}{Theorem}
\newtheorem*{theorem*}{Theorem}

\newtheorem*{conjecture*}{Conjecture}
\newtheorem*{belief*}{General Belief}
\newtheorem{corollary}[theorem]{Corollary}

\newtheorem*{definition}{Definition}

\newtheorem{lemma}[theorem]{Lemma}
\newtheorem*{notation}{Notation}

\newtheorem{proposition}[theorem]{Proposition}

\newtheorem{bigtheorem}{Theorem}

\renewenvironment{proof}[1][Proof]{\noindent\textbf{#1.} }{\ \rule{0.5em}{0.5em}}
\newenvironment{acknowledgements}[1][Acknowledgements]{\noindent\textbf{#1.} }{\ }

\newcommand{\Z}{\mathbb{Z}}
\newcommand{\R}{\mathbb{R}}
\newcommand{\N}{\mathbb{N}}
\newcommand{\Q}{\mathbb{Q}}
\newcommand{\Orb}{\mathcal{O}}

\newcommand{\floor}[1]{\left\lfloor {#1}\right\rfloor}
\newcommand{\ceil}[1]{\left\lceil {#1}\right\rceil}

\newcommand{\Rf}{\mathbf{R}_{\lfloor\cdot\rfloor}}
\newcommand{\Rc}{\mathbf{R}_{\lceil\cdot\rceil}}

\renewcommand{\mod}{\ \mathrm{mod}\ }

\newcommand{\T}{\mathcal{T}}
\newcommand{\proj}{\mathrm{proj}}
\newcommand{\id}{\mathrm{id}}
\newcommand{\poly}{\mathrm{polynomial}}

\begin{document}

\title{\textbf{A Minimal Subsystem of the Kari-Culik Tilings}}
\author{Jason Siefken\footnote{Supported by the University of Victoria and NSERC}\\
 \small Department of Mathematics and Statistics\\
 \vspace{-.4em}
 \small University of Victoria\\
 \vspace{-.4em}
 \small P. O. Box 3045\\
 \vspace{-.4em}
 \small Victoria, BC \\
 \vspace{-.4em}
 \small \textsc{Canada} V8P 5C2\\
 \small siefkenj@uvic.ca}
\date{\today}
\maketitle

\begin{abstract}
	The Kari-Culik tilings are formed from a set of 13 Wang tiles that tile
	the plane only aperiodically.  They are the smallest known set of Wang tiles
	to do so and are not as well understood as other examples of aperiodic Wang tiles.
	We show that the $\Z^2$ action by translation on 
	a certain subset of the Kari-Culik tilings, namely those
	whose rows can be interpreted as Sturmian sequences (rotation sequences), is minimal.
	We give a characterization of this space as a skew product as well as 
	explicit bounds on the waiting time between occurrences of $m\times n$
	configurations.
	\let\thefootnote\relax\footnotetext{\emph{2000 Mathematics Subject Classification.} 
	Primary: 52C20, 05B45, 5223.  Secondary: 37E05.}
	\let\thefootnote\relax\footnotetext{\emph{Key words:} aperiodic tiling, Wang tiles, aperiodic SFT, Kari-Culik tiling. }
\end{abstract}

\section{Introduction}

	The Kari-Culik tilings are tilings of the plane by the
	13 square tiles specified in Figure \ref{FigKCTiles}.
	Translations, but not rotations, of tiles are allowed, and
	two tiles may share an edge if the 
	labels of those edges match.

	\begin{figure}[h!]
		\footnotesize
		\begin{center}
		\includegraphics[width=5in]{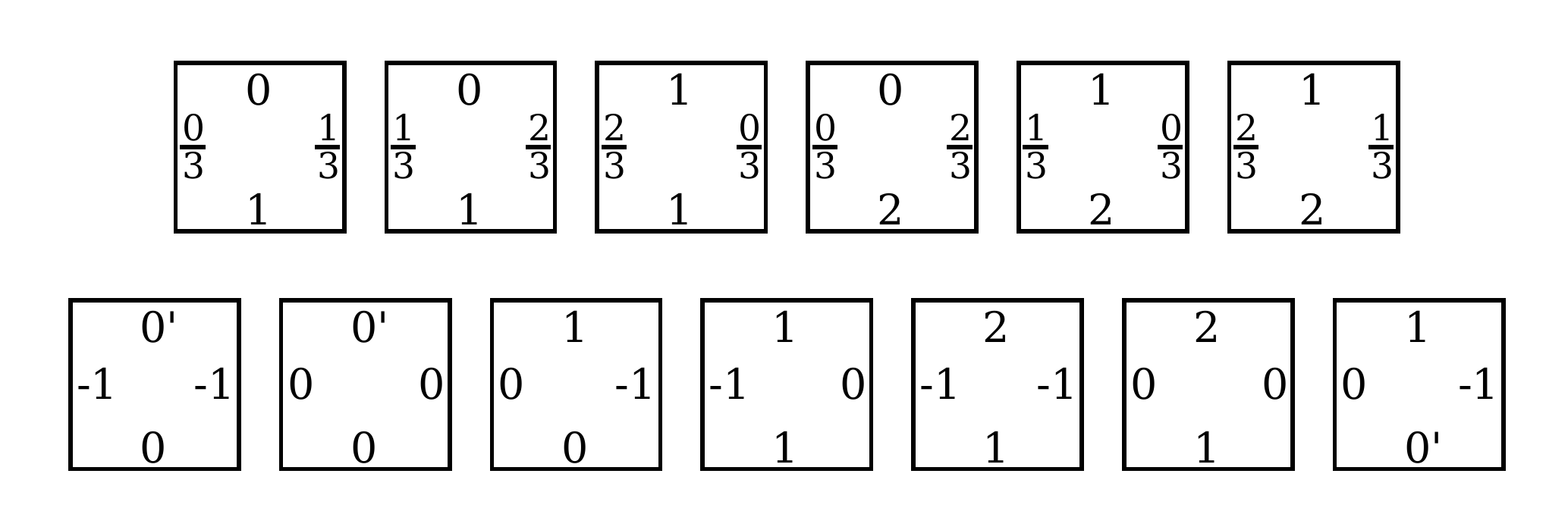}
		\end{center}
		\caption{ \footnotesize List of the 13 Kari-Culik tiles.}
		\label{FigKCTiles}
	\end{figure}

	Published in 1995, they are the smallest known set
	of square tiles admitting only aperiodic tilings of the plane.  Further,
	unlike the Robinson tilings and other well-known aperiodic tilings, the 
	current proofs of aperiodicity
	for the Kari-Culik tilings are based on number-theoretic arguments
	with no known hierarchical explanation.  

	Eigen, Navarro, and Prasad provide a detailed exposition and proof of the
	aperiodicity and existence of Kari-Culik tilings \cite{eigen}.  In their proof,
	Eigen et al\mbox{.} show the existence of Kari-Culik tilings arising from 
	Sturmian sequences (defined in the next section).
	It was unknown if there were Kari-Culik configurations not arising from Sturmian
	sequences until Durand, Gamard, and Grandjean  showed that the set of all Kari-Culik
	tilings has positive topological entropy \cite{durand}.  Since the 
	number of Sturmian sequences of length $n$ is bounded by a polynomial $p(n)$ \cite{berenstein}, 
	the fact that $n\log p(n)/n^2\to 0$ shows that the
	subset of
	tilings arising from Sturmian sequences must have zero-entropy. This implies the
	existence of non-Sturmian Kari-Culik tilings.

	In this paper, however, we will only concern ourselves with a subset
	of the Kari-Culik tilings.  Let $KC$ be the subset of the Kari-Culik tilings
	whose rows form (generalized) Sturmian sequences.

	\begin{bigtheorem}\label{ThmKCConjSkew}
		The $\Z^2$ action by translation on $KC$ is conjugate to
		a skew product acting on the space $[1/3,2]\times \varprojlim \R/(6^n\Z)$.
	\end{bigtheorem}

	\begin{bigtheorem}\label{ThmKCMiminal}
		The $\Z^2$ action by translation on $KC$ is minimal.
	\end{bigtheorem}

	\begin{bigtheorem}\label{ThmKCBounds}
		If $A$ is an $m\times n$ sub-configuration
		of some point in $KC$, every sub-configuration of size
		at least $6^{34.464m+25}n^{34.464}\times 6^{5m+3}n^4$ of a point
		in $KC$ contains a copy of $A$.
	\end{bigtheorem}
	

	Theorem \ref{ThmKCBounds} gives explicit bounds on the maximum
	waiting time for any $m\times n$ configuration.  In particular 
	these bounds are of the form $\exp(m)\cdot\poly(n)$. 
	
\section{Definitions}
	A Wang tiling is a class of nearest-neighbor subshift of finite type (SFT)
	on $\Z^2$ whose rules can be given by square tiles with labeled edges.  
	We will interpret valid Kari-Culik tilings
	as a nearest-neighbor SFT on 13 symbols with the adjacency rules
	given by Figure \ref{FigKCTiles}.  The set of the 13 Kari-Culik tiles
	will be called $\mathfrak K$.

	$T,S$ are the usual horizontal and vertical shifts on $\Z^2$.  We say
	a configuration $x\in\mathfrak K^{\Z^2}$ is \emph{aperiodic} if 
	$T^aS^b x = x$ implies that $(a,b)=(0,0)$. 
	We denote by $\Orb(x) = \{T^aS^bx: a,b\in\Z\}$
	the orbit of $x$ and $\Orb_X(x) = \{X^ax:a\in\Z\}$ the orbit of $x$ under
	the map $X$ where $X\in\{T,S\}$.

	Given a vector $\vec x\in\R^\Z$,
	$x_i$ is the $i$th component of $\vec x$.
	If $x=(\ldots, x_{-1},x_{0},x_{1},\ldots)\in\mathfrak K^{\Z}$,
	then $(x)_i^j = (x_i,\ldots, x_j)$ is the subword of $x$ from position
	$i$ to $j$ and if $x\in \mathfrak K^{\Z^2}$, $(x)_i$ refers to the $i$th
	row of $x$. Further, if $A\subset \Z^2$ and $x\in \mathfrak K^{\Z^2}$,
	then $x|_A$ is the restriction of $x$ to the indices in $A$.  
	We denote by $d$ the usual metric on bi-infinite sequences. That is,
	\[
		d(x,y) = \inf\{ 2^{-i}: (x)_{-i}^i=(y)_{-i}^i\}.
	\]

	\begin{definition}[$a$-valuation]
		For $q\in \Q$, we denote by $|q|_a$ the $a$-valuation of $q$.
		That is, if
		\[
			q=\prod_{p\text{ prime}} p^{n_p}
		\]
		is the prime decomposition of $q$ with $n_p\in\Z$, then
		$
			|q|_a = -n_a.
		$
	\end{definition}

	\begin{definition}[Rotation Sequence]
		A \emph{rotation sequence} corresponding to
		the parameters $\alpha, t\in \R$ is the sequence $x=\Rf(\alpha,t)$
		or $x'=\Rc(\alpha,t)$ where
		\[
			(x)_i = \floor{i\alpha+t}-\floor{(i-1)\alpha+t}
		\]
		and
		\[
			(x')_i = \ceil{i\alpha+t}-\ceil{(i-1)\alpha+t}.
		\]
		The parameters $\alpha,t$ are called the \emph{angle} and
		the \emph{phase} of the sequence.
	\end{definition}

	\begin{definition}[Sturmian Sequence]
		A sequence $x$ is a \emph{Sturmian sequence} if $x=\Rf(\alpha,t)$
		or $x=\Rc(\alpha,t)$ for some $\alpha,t\in \R$. We call
		$\alpha$ the \emph{angle} of $x$ and $t$ a \emph{phase} of $x$.
		$\mathcal S$ denotes the set of all Sturmian sequences
		and $\bar{\mathcal S}$ denotes its closure under $d$.
		$\bar{\mathcal S}$ is called the set of \emph{generalized
		Sturmian sequences}.
	\end{definition}
	Unlike some authors, we allow Sturmian sequences to be periodic.
	A consequence is that although the angle of any Sturmian sequence
	is uniquely defined as the average of its symbols, if the angle of a
	Sturmian sequence is rational (and hence the sequence is 
	periodic), its phase is not uniquely defined.
	Further, we
	allow Sturmian sequences to consist of symbols other than $\{0,1\}$.
	For a detailed
	exposition of Sturmian sequences and their equivalent characterizations, 
	see \cite{fogg,lothaire}.  It is also worth noting that the set of generalized
	Sturmian sequences is strictly larger than the set of Sturmian sequences.
	For example $x=(\ldots, 0,0,1,0,0,\ldots)\in\bar{\mathcal S}$ but $x\notin \mathcal S$.

	\begin{definition}[Balanced Sequence]
		A sequence $x$ is a \emph{balanced sequence} if for any two subwords
		$u=u_1\cdots u_n$ and $v=v_1\cdots v_n$ of $x$ of equal length, we have
		$|\sum u_i - \sum v_i| \leq 1$.
	\end{definition}

	Every Sturmian sequence is a balanced sequence \cite{fogg}. 
	\begin{proposition}
		Every generalized Sturmian sequence is a balanced sequence.
	\end{proposition}
	\begin{proof}
		Let $x\in \bar{\mathcal S}$ and suppose there exist $u=u_1\cdots u_n$ and $v=v_1\cdots v_n$
		such that $|\sum u_i-\sum v_i|>1$.  Since $x$ is a limit point of $\mathcal S$, there
		must exist $x'\in \mathcal S$ such that $u,v$ are both subwords of $x'$.  However, $x'$ is a
		Sturmian sequence and therefore balanced, which is a contradiction.
	\end{proof}

	\begin{definition}
		For a sequence $x$, define $\alpha(x)=\lim_{n\to\infty}\frac{1}{2n+1}\sum_{|i|\leq n} x_i$ to be the average of
		its symbols if it exists.
	\end{definition}

	For any $x\in \bar{\mathcal S}$, $\alpha(x)$ is defined. This follows from the
	fact that  $x\in\bar{\mathcal S}$ is a balanced sequence
	and therefore the average of its symbols exist.
	Now, from the definition of a rotation sequence, we see if
	$x\in \mathcal S$, then $\alpha(x)$ is the angle of $x$.

	\begin{definition}
		For a Sturmian sequence $x$, define $t(x)=\inf\{t:t\text{ is a phase for }x\}$.
	\end{definition}

	\begin{proposition}
		For $x\in\bar{\mathcal S}$, if $\alpha(x)\notin \Q$ then $x\in \mathcal S$
		and
		$x=\Rf(\alpha(x),t(x))$ or $x=\Rc(\alpha(x),t(x))$.
	\end{proposition}
	\begin{proof}
		Since $\alpha(x)\notin \Q$, we necessarily have that $x$ is not eventually periodic.  Since $x$
		is also a balanced sequence, $x$ is an aperiodic Sturmian sequence \cite{fogg}.
		The proposition now follows from the fact that an irrational
		Sturmian sequence has a unique phase.
	\end{proof}

	\begin{definition}
		$\Phi:\mathfrak K^{\Z^2}\to \{0,1,2\}^{\Z^2}$ is projection onto
		the bottom labels of tiles in $\mathfrak K$ followed by mapping 
		the symbol $0'$ to $0$.
	\end{definition}

	\begin{definition}
		The set $KC=\{x: x\text{ is a Kari-Culik tiling and }\Phi(x)\text{
		consists of generalized}\allowbreak\text{ Sturmian rows}\}$.
		$KC_{\Q^c}=\{x\in KC: (\Phi(x))_i\text{ has an irrational angle for all }i\}$.
	\end{definition}

	Extending notation, for $x\in KC$, we define 
	$\alpha(x) = (\ldots, \alpha((\Phi(x))_0),\alpha((\Phi(x))_1)
	,\ldots)$ and $t(x) = (\ldots, t((\Phi(x))_0),t((\Phi(x))_1),
	\ldots)$ to be the \emph{angle vector} and \emph{phase vector} of $x$ (if
	they exist).

\section{Kari-Culik Properties}
	Notice that for any Kari-Culik tiling, the rows fall into two distinct categories:
	those where every tile
	has left-right edge labels in $\{\tfrac{0}{3},\tfrac{1}{3},\tfrac{2}{3}\}$ and
	those where every tile has left-right edge labels in $\{0,-1\}$.  
	We will call these rows as well as the tiles in each row
	\emph{type $\frac{1}{3}$} and \emph{type $2$} respectively.  
	The convention in this paper will be to refer to the labels of a tile in $\mathfrak K$
	in clockwise order starting with the bottom label.  That is, the labels $a,b,c,d$
	of a tile will correspond to the figure:
		\begin{center}
			\includegraphics[width=.75in]{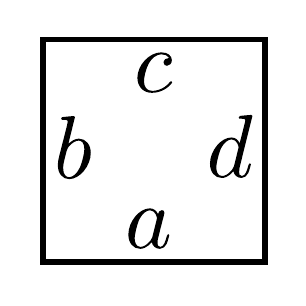}
		\end{center}

	Part of the cleverness of the Kari-Culik
	tilings is that every tile satisfies the following.

	\begin{definition}[Multiplier Property]
		A Kari-Culik tile with bottom, left, top, and right labels of
		$a,b,c,d$ satisfies the relationship
		\begin{equation}\label{EqMultProperty}
			\lambda a + b = c + d
		\end{equation}
		where $\lambda\in\{\frac{1}{3},2\}$ corresponds to the type of the tile.
		We also refer to $\lambda$  as the \emph{multiplier} of the tile.
	\end{definition}

	\begin{proposition}\label{PropMultiplier}
		Fix a Kari-Culik configuration $x$ and let $r_0=\Phi((x)_0)$ and
		$r_1 = \Phi((x)_1)$.  Then, if the average of $r_0$ exists,
		it satisfies the relation
		\[
			\lambda \alpha(r_0) = \alpha(r_1)
		\]
		where $\lambda\in\{\frac{1}{3},2\}$ is the type of $(x)_0$.
	\end{proposition}
	\begin{proof}
		This is a direct result of the telescoping nature of the multiplier property
		when rewritten as $\lambda a-c=d-b$.
		Notice that in any row, every tile is the same type and therefore has the
		same multiplier.
		Let $a_i$ be the bottom labels and $c_i$ be the top labels of $(x)_0$.
		Summing along a central segment of length $2n+1$, we have
		\begin{equation}\label{EqMult}
			\lambda \sum_{i=-n}^n a_i -\sum_{i=-n}^n  c_i = d -b
		\end{equation}
		where $b,d$ are the left
		and right labels of the central segment.
		Since
		\[
			\alpha(r_0) = \lim_{n\to\infty}\frac{1}{2n+1}\sum_{i=-n}^n a_i
			\qquad\text{and}\qquad
			\alpha(r_1) = \lim_{n\to\infty}\frac{1}{2n+1}\sum_{i=-n}^n c_i
		\]
		and $b,d$ are bounded, dividing both sides of Equation 
		\eqref{EqMult} by $2n+1$
		and taking a limit produces the desired relationship.
	\end{proof}

	It was further shown by Durand et al\mbox{.} that in fact the average of the
	bottom labels of any row in a Kari-Culik tiling exists, making 
	the assumption that the average exists in Proposition 
	\ref{PropMultiplier} unnecessary \cite{durand}.

	\begin{proposition}
		For a Kari-Culik tiling $x$, $\alpha((\Phi(x))_i)\in [1/3,2]$.
	\end{proposition}
	\begin{proof}
		As noted earlier, $\alpha((\Phi(x))_i)$ always exists.
		Fix $i$ and let $\alpha = \alpha((\Phi(x))_i)$. Inspecting
		the tile set, we see that the largest symbol on the bottom 
		of any tile is $2$ and so $\alpha\leq 2$.  To see that $\alpha\geq 1/3$,
		we will consider rows by type.  For a row of type $\frac{1}{3}$, the smallest
		symbol appearing on the bottom is $1$, and so $\alpha\geq 1\geq 1/3$.

		For a row of type $2$, notice that the bottom labels may contain $0$ or
		$0'$ but not both (if a row of type $2$ had both $0$ and $0'$ on the bottom, the
		row below it would need to have tiles of both type $\frac{1}{3}$ and type $2$).
		If the bottom labels only contain $0$, then the row below $(x)_i$ must 
		be of type $\frac{1}{3}$.
		Inspecting the type $\frac{1}{3}$ tiles, we see that no more than 
		two consecutive $0$ symbols may occur as top labels and so $(x)_i$
		cannot have more than two $0$ symbols in a row as bottom labels giving
		$\alpha\geq 1/3$.  Finally, notice that as bottom labels, all occurrences
		of $0'$ are isolated.  Thus, if the row below $(x)_i$ is of type $2$,
		$\alpha\geq 1/2\geq 1/3$.
	\end{proof}

	\begin{definition}
		For $x\in[1/3,2]$, define
		\[
			\lambda_x = \left\{\begin{array}{cl}
					2 &\text{ if } x\in[1/3,1)\\
					1/3 &\text{ if } x\in [1,2]
				\end{array}\right.
			\qquad\text{and}\qquad
			f(x) = \lambda_x x = \left\{\begin{array}{cl}
					2x &\text{ if } x\in[1/3,1)\\
					x/3 &\text{ if } x\in [1,2]
				\end{array}\right. .
		\]
	\end{definition}

	\begin{corollary}\label{PropAnglesRespect}
		Fix a Kari-Culik configuration $x$ and let $r_i=\Phi((x)_i)$. Then,
		\[
			\alpha(r_{i+1}) = f(\alpha(r_{i}))
		\]
		provided $\alpha(r_i)\neq 1$.
	\end{corollary}
	\begin{proof}
		Since $\alpha(r_{i+1}) = \lambda\alpha(r_{i})$ for some 
		$\lambda\in\{\frac{1}{3},2\}$,
		the constraint that both $\alpha(r_{i+1}),\alpha(r_{i})
		\in[1/3,2]$ uniquely determines $\lambda$ when 
		$\alpha(r_i)\neq 1$. 
	\end{proof}

	Even if $\alpha(r_i)=1$, there are still only two options
	for $\alpha(r_{i+1})$ and as will be
	shown in Proposition \ref{PropfAperiodic}, orbits under $f$
	are aperiodic, ensuring
	this can occur at most once.

	Using observations about the multiplier of tiles and the 
	averages of sequences of bottom labels, we can refine our classification
	of rows of the Kari-Culik tilings.

	\begin{definition}
		Let $x\in KC$ and $r_i=(x)_i$ be the $i$th row of $x$.
		We define the \emph{general type} of $r_i$ based on the tiles in
		$r_{i-1},r_i,r_{i+1}$ in the following way.
			
		\emph{Type $\frac{1}{3}$}: $r_i$ is of 
			general type $\frac{1}{3}$ if $r_i$ 
			consists of type $\frac{1}{3}$ tiles.

		\emph{Type $2.1$}: $r_i$ is of 
			general type $2.1$ if $r_i$ 
			consists of type $2$ tiles and $r_{i+1},r_{i-1}$
			both consist of type $\frac{1}{3}$ tiles.

		\emph{Type $2.2t$}: $r_i$ is of 
			general type $2.2t$ if $r_i$ 
			consists of type $2$ tiles and $r_{i+1}$
			consists of type $\frac{1}{3}$ tiles while $r_{i-1}$
			consists of type $2$ tiles.

		\emph{Type $2.2b$}: $r_i$ is of 
			general type $2.2b$ if $r_i$ 
			consists of type $2$ tiles and $r_{i-1}$
			consists of type $\frac{1}{3}$ tiles while $r_{i+1}$
			consists of type $2$ tiles.

		We consider a pair of rows whose top row is of general type $2.2t$ and
		whose bottom row is of general type $2.2b$ as \emph{type $2.2$}.
	\end{definition}

	Since general type $\frac{1}{3}$ exactly corresponds to type $\frac{1}{3}$
	and we have no previous definition for type $2.1$, $2.2t$, $2.2b$, or $2.2$,
	without ambiguity we may from now on refer to the general type of a row
	as simply the type of that row.

	\begin{proposition}\label{PropTypesOfRows}
		Let $x\in KC$ and $r_i=(x)_i$ be the $i$th row of $x$.
		The general type of $r_i$ is unique and the tiles that may
		appear in $r_i$ are contained in exactly one
		of the following (non-disjoint) sets based on general type.

		Type $\frac{1}{3}$:
		\vspace{-3.5em}
		\begin{center}
			\includegraphics[width=4.2in]{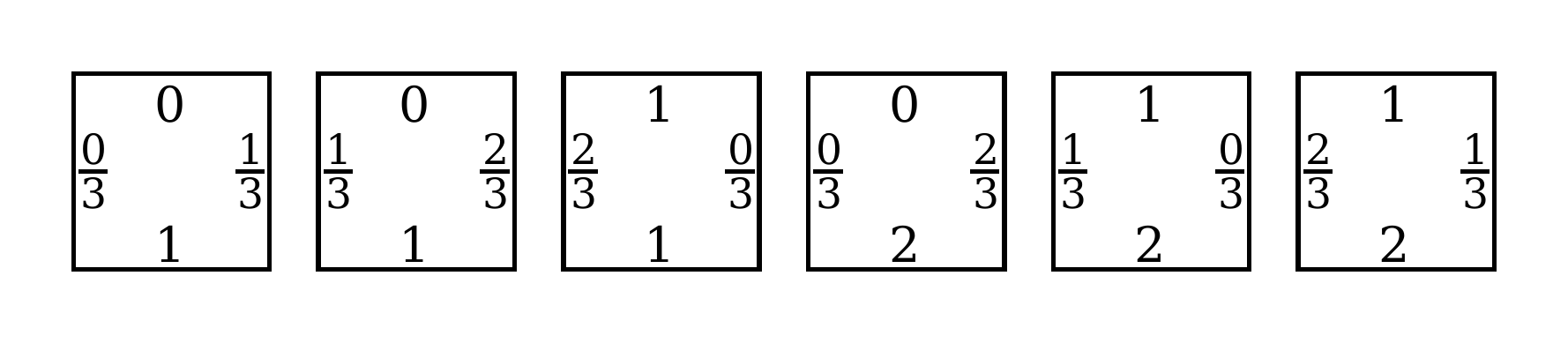}
		\end{center}
		
		Type $2.1$: 
		\vspace{-3.5em}
		\begin{center}
			\includegraphics[width=3in]{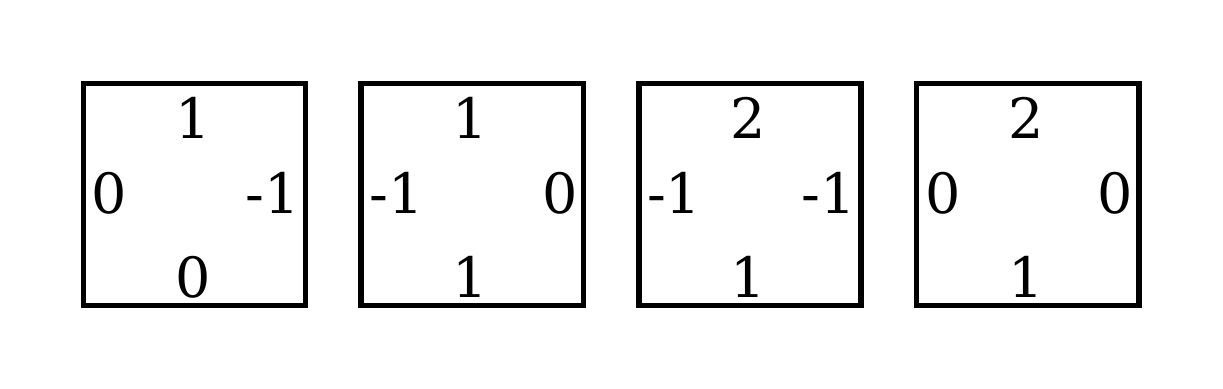}
		\end{center}
		\vspace{-.6cm}
		
		Type $2.2t$: 
		\vspace{-1.5em}
		\begin{center}
			\includegraphics[width=3in]{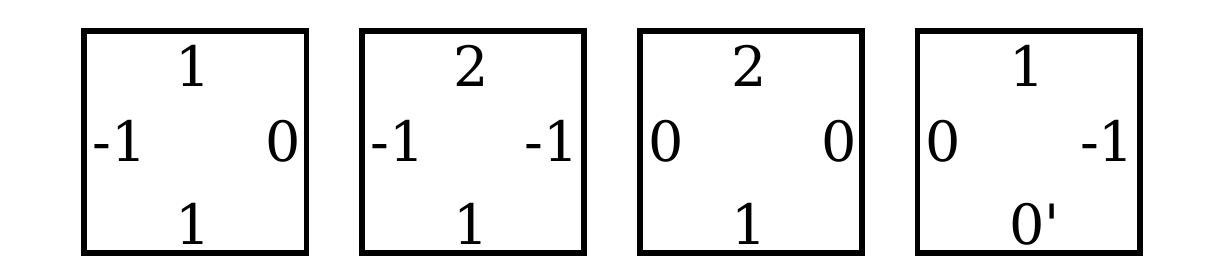}
		\end{center}
		
		Type $2.2b$:
		\vspace{-1.5em}
		\begin{center}
			\includegraphics[width=3in]{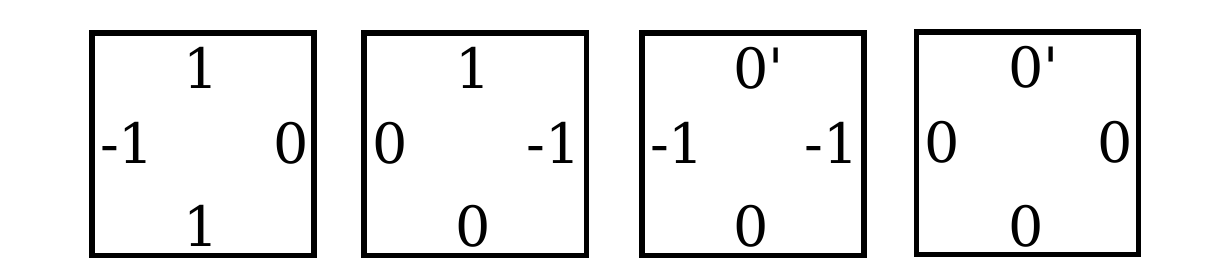}
		\end{center}

		A pair of rows whose top tiles are
		type $2.2t$ and bottom tiles are type $2.2b$ taken together
		and considered as type $2.2$ consists of the stacked tiles:
		\vspace{-1em}
		\begin{center}
			\includegraphics[width=3in]{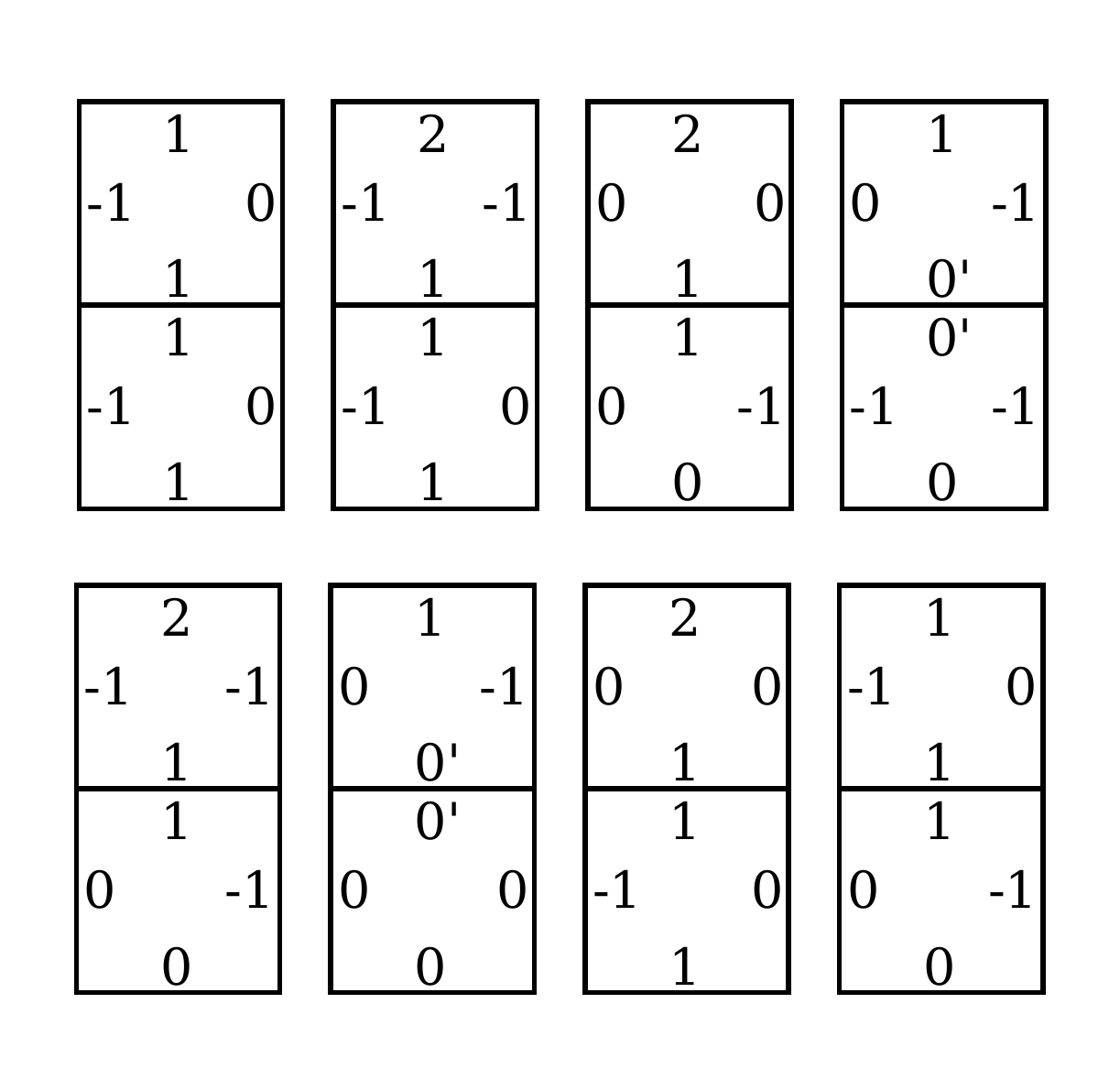}
		\end{center}
	\end{proposition}
	\begin{proof}
		Fix a Kari-Culik tiling $x$. Let $r_i=(x)_i$ be the $i$th row of $x$
		and let 
		$\lambda_i$ be its multiplier.
		Let $\alpha$ be the average of the bottom labels of $r_i$.
		We will show something slightly stronger than is asked,
		namely that except for $\alpha\in\{1/2,2/3,1\}$, $\alpha$
		uniquely
		determines the type $r_i$.

		If $\alpha\in(1,2]$, then
		$\lambda = \frac{1}{3}$, and so $r_i$ must consist of tiles of 
		type $\frac{1}{3}$, making $r_i$ of general type $\frac{1}{3}$.

		If $\alpha\in (1/2,2/3)$, $r_{i+1},r_{i-1}$ must be of
		type $\frac{1}{3}$, and so $0'$ cannot occur as a label, 
		making $r_i$ of general type $2.1$ and
		leaving the only available
		tiles those listed as type 2.1.

		If $\alpha\in [1/3,1/2)$, the rows $r_i$ and $r_{i+1}$
			are both of type 2 and $r_{i-1}$ and $r_{i+2}$ 
			are of type $\frac{1}{3}$.  Thus, $r_i$ must
			be of general type $2.2b$ and must consist
		of the tiles listed as type $2.2b$.

		Finally, if $\alpha\in (2/3,1)$, $r_{i-1}$ is of type $2$
		and $r_{i+1}$ and $r_{i-2}$ are both of type $\frac{1}{3}$.
		Thus, $r_i$ must be of general type $2.2t$
		and consists of the tiles listed as type $2.2t$.

		The tiles listed as type 2.2 consist of the
		ways to stack type 2 tiles to be compatible on tops and bottoms with type
		$\frac{1}{3}$ tiles and so correspond exactly to the cases where
		$2.2t$ and $2.2b$ tiles arise in consecutive rows.

		In the remaining cases of $\alpha\in\{1/2,2/3,1\}$, the type of
		$r_i$ is not strictly determined by $\alpha$, but nonetheless,
		the tiles in $r_i$ fall into one of the four categories and the 
		classification is unique.
	\end{proof}

	The tiles listed as type $2.1$ have non-trivial intersection
	with the tiles listed as type $2.2t$ and type $2.2b$, 
	however since no two rows of type $2.1$ occur consecutively, the 
	categorization is unique.
	We call the pairs of tiles listed as type 2.2 \emph{stacked tiles}.
	When we think of a row of a Kari-Culik tiling as being type
	2.2, we may think of its multiplier as being 4 (since it is composed
	of two consecutive rows with multiplier 2).

	\begin{proposition}[Liousse \cite{liousse}]\label{PropfAperiodic}
		The map $f$ is conjugate to an irrational rotation by
		$\log 2 / \log 6$.
	\end{proposition}
	\begin{proof}
		An explicit conjugacy 
		$\phi:[1/3,2]\to[0,1]$ is given by $\phi(x) = \frac{\log x + \log 3}{\log 6}$.
	\end{proof}

	Analogously to the way Eigen et al\mbox{.} show the existence of Kari-Culik
	tilings through their Basic Tile Construction, we will show how to take
	an angle vector and a phase vector and produce a valid element of $KC$.

	\begin{definition}[BC Property]
		A pair of vectors $(\vec \alpha, \vec t)\in[1/3,2]^\Z\times [0,1]^\Z$
		satisfies the \emph{BC property} (Basic Construction property) if
		\[
			\lambda_i=\frac{\alpha_i}{\alpha_{i+1}}\in\{\tfrac{1}{3},2\}
		\]
		and\[
			\begin{array}{cl}
				2 t_i = t_{i+1} \mod 1&\text{ if }\lambda_i=2\\
				t_i = 3 t_{i+1} \mod 1&\text{ if }\lambda_i=\tfrac{1}{3}
			\end{array}
		\]
		for all $i$.
	\end{definition}

	Given a pair of vectors $(\vec \alpha,\vec t)$ satisfying the BC property,
	we can construct a point $y\in KC$ via the following procedure.  The
	tile at position $m,n$ in $y$ has bottom, left, top, and right edges given by
	\[
		\begin{array}{rcl}
			\text{bottom} & = & \floor{n\alpha_m+t_m}-\floor{(n-1)\alpha_m+t_m}\\
			\text{left} & = & \lambda\floor{(n-1)\alpha_m+t_m}-\floor{\lambda(n-1)\alpha_m+t_m}\\
			\text{top} & = & \floor{n\alpha_{m+1}+t_{m+1}}-\floor{(n-1)\alpha_{m+1}+t_{m+1}}\\
			\text{right} & = & \lambda\floor{n\alpha_m+t_m}-\floor{\lambda n\alpha_m+t_m} 
		\end{array}
	\]
	where $\lambda = \alpha_m/\alpha_{m+1}$.  Further, if either the bottom or the top label
	is computed to be $0$, then $0$ is replaced with $0'$ if $\alpha_{m-1}\in[1/3,1/2]$ 
	(respectively $\alpha_{m}\in[1/3,1/2]$).
	We can also do the same construction using $\ceil{\cdot}$ instead of $\floor{\cdot}$.
	We call a tiling constructed in this way a \emph{Basic Construction}
	with parameters $(\vec \alpha, \vec t)$.

	\begin{proposition}[Robinson \cite{robbie}]\label{PropRobbie}
		If $(\vec \alpha, \vec t)$ satisfies the BC property, then the resulting
		Basic Construction using either $\floor{\cdot}$ or $\ceil{\cdot}$
		is an element of $KC$.
	\end{proposition}
	\begin{proof}
		First observe that if $y$ is the result of a Basic Construction, 
		then $\Phi(y)$ consists of rows that are rotation sequences and
		therefore Sturmian.  Further, by definition, the top labels of
		each row of $y$ are guaranteed to be compatible with the bottom labels
		of the next row,
		and the right labels of each column of $y$ are guaranteed to be compatible
		with the left labels of the next column.

		The remainder of the proof involves checking for all ranges of
		$\alpha_m,t_m$ that the resulting bottom, left, top, and right labels
		correspond to an actual tile in $\mathfrak K$.  The details of this
		are straightforward, and after substituting $t_{m+1}=2t_m\mod 1$
		or $t_m=3t_{m+1}\mod 1$ depending on the ratio $\alpha_m/\alpha_{m+1}$,
		it
		requires only examining what cases result from the choice of $\alpha_m$,$t_m$
		or $\alpha_m,t_{m+1}$.
	\end{proof}

	Proposition \ref{PropBasicConstructRepr} provides a partial converse to Proposition
	\ref{PropRobbie}. Proving that a tiling in $KC_{\Q^c}$,
	like tilings arising from a Basic Construction, can be expressed
	with one of $\Rf$ or $\Rc$, but never requires a mixture of both,
	is the bulk of the proof of Proposition \ref{PropBasicConstructRepr}.

	\begin{proposition}\label{PropBasicConstructRepr}
		If $y\in KC_{\Q^c}$ and $(\vec \alpha, \vec t)=(\alpha(\Phi(y)),t(\Phi(y)))$
		are the angle and phase vectors of $y$, then $y$ is the result of a
		Basic Construction arising from $(\vec \alpha, \vec t)$
		using either $\floor{\cdot}$ or $\ceil{\cdot}$.
	\end{proposition}
	\begin{proof}
		First note that since $y\in KC_{\Q^c}$, the rows of $\Phi(y)$
		are Sturmian sequences and therefore rotation sequences (since
		rows in $\bar {\mathcal S}\backslash \mathcal S$ are excluded).  
		
		We will first show that $(\vec \alpha,\vec t)$
		satisfies the BC property.  Fix $k\in\Z$.  
		Since Corollary \ref{PropAnglesRespect}
		already shows that $\vec \alpha$ is determined by $f$ 
		and $\alpha_k$ (that is $\alpha_{k+i} = f^i(\alpha_k)$),
		we only need to show that either $t_{k+1}=2t_k\mod 1$
		or $t_k=3 t_{k+1}\mod 1$ in accordance with $\alpha_k$.  
		For simplicity, call $\alpha=\alpha_k$, $t=t_k$, and $t'=t_{k+1}$.
		Let $\lambda = \alpha_k/\alpha_{k+1}$ be the type of the $k$th row
		of $y$ and let $a_i,b_i,c_i,d_i$ be the bottom, left, top, and right
		labels of the $i$th tile in $(y)_{k}$.
		We divide the proof into two similar cases depending on $\lambda$.

		Case $\lambda=1/3$:
		We will
		assume the Sturmian sequences $(\Phi(y))_k$ and $(\Phi(y))_{k+1}$
		may both be represented using $\Rf$, but note that for every combination
		$(\Rf,\Rf)$, $(\Rf,\Rc)$, $(\Rc,\Rf)$, and $(\Rc,\Rc)$ of ways to represent
		$(\Phi(y))_k$ and $(\Phi(y))_{k+1}$, upon replacing $\floor{\cdot}$
		with $\ceil{\cdot}$ where appropriate, the same argument still works.
		By Corollary \ref{PropAnglesRespect}, $\alpha_{k+1} = \lambda\alpha_k=\lambda
		\alpha$,
		and so
		we have the following relationship for the bottom and top labels:
		\[
			a_i = \floor{i\alpha+t} - \floor{(i-1)\alpha+t}
			\qquad\text{and}\qquad
			c_i = \floor{\frac{i\alpha+3t'}{3}} - \floor{\frac{(i-1)\alpha+3t'}{3}}.
		\]
		Exploiting the telescoping nature of the Multiplier Property
		(shown in Equation \eqref{EqMult}) and summing from $i=n+1$ to $m$, we get
		\begin{equation}\label{EqPhase}
			(b_n-d_m) + \frac{1}{3}\left(
				\floor{\alpha m+t} - \floor{\alpha n+t}
			\right)
			=
			 \floor{\frac{\alpha m+3t'}{3}} - \floor{\frac{\alpha n+3t'}{3}}.
		\end{equation}

		Since $\alpha\notin \Q$, we can pick $n,m$ so that
		\[
			\frac{\alpha m+3t'}{3}=k_m + \varepsilon_m\qquad
			\text{and}\qquad
			\frac{\alpha n+3t'}{3}=k_n - \varepsilon_n
		\]
		where $k_m,k_n\in\Z$ and $\varepsilon_m,\varepsilon_n$ are arbitrarily
		small positive numbers.  Upon this choice, the right side of Equation
		\eqref{EqPhase} simplifies to $k_m-k_n+1$.
		By rearranging and substituting into Equation
		\eqref{EqPhase}, we get
		\[
			\floor{3(k_m+\varepsilon_m) + (t-3t')}
			-\floor{3(k_n-\varepsilon_n) + (t-3t')}
			=3(k_m-k_n)+3-3(b_n-d_m),
		\]
		but since $b_n-d_m$ is bounded above by $2/3$, we conclude
		\begin{equation}\label{EqContra}
			\floor{3(k_m+\varepsilon_m) + (t-3t')}
			-\floor{3(k_n-\varepsilon_n) + (t-3t')}
			\geq
			3(k_m-k_n)+1.
		\end{equation}
		If $(t-3t'\mod 1)=\gamma\neq 0$, choosing 
		$\varepsilon_n,\varepsilon_m<\!\!< \gamma$
		gives a contradiction (with the left hand side 
		of Equation \eqref{EqContra} yielding $3(k_m-k_n)$).
		Thus, $t=3t'\mod 1$.

		Case $\lambda=2$: Since this case is 
		nearly identical to the $\lambda=1/3$ case, we will omit the details,
		noting only that in this case
		the relationship between bottom labels and top 
		labels is reversed.  That is, in this case fix $\alpha=\alpha_{k+1}$
		and rewrite $\alpha_k=\alpha/2$. 
		
		We have now shown that $(\vec \alpha, \vec t)$ satisfies the BC property.
		To complete the proof and show that $y$ arises as a Basic Construction,
		we need to show that every Sturmian sequence in $\Phi(y)$ can
		be written with exclusively $\Rf$ or exclusively $\Rc$.

		Notice that since every component of $\vec \alpha$ is rationally 
		related to $\alpha_0$, we have either $\vec \alpha\in \Q^\Z$
		or $\vec \alpha\in (\Q^c)^\Z$.
		By assumption however,  $\vec \alpha\in (\Q^c)^\Z$ and so
		$n\vec \alpha + \vec t\in\Q^\Z$ for at most one $n$.
		Define
		\[
			r_{i,n} = n\alpha_i+t_i.
		\]
		Suppose that the $i$th row of $\Phi(y)$ requires
		$\Rf$ or $\Rc$ to be written as a rotation sequence.
		This implies that for some $n$ we have
		$r_{i,n}\in\Z$. Fix this $n$. By our observation that 
		$m\vec \alpha + \vec t\in\Q^\Z$ for at most one $m$, we may conclude
		that $r_{j,n'}\notin \Z$ for any $n'\neq n$ and $j\in \Z$.

		Let
		\[
			B=\{i:(\Phi(y))_i\text{ requires $\Rf$ or $\Rc$}\},
		\]
		and notice again that by the uniqueness of $n$ (with $n$ still
		being fixed as above),
		$B=\{j:r_{j,n}\in \Z\}$.
		We will now show that $B$ consists of a contiguous sequence
		of integers.  
		
		Exploiting the fact that $(\vec \alpha,\vec t)$
		satisfies the BC property, we may conclude $r_{j+1,n}=2r_{j,n}$
		or $r_{j+1,n}=\frac{1}{3}(r_{j,n}+i)$ where $i\in\{0,1,2\}$. This
		implies that if $|r_{j,n}|_3>0$ then $|r_{j+1,n}|_3>0$ where $|\cdot|_3$
		is the $3$-valuation.

		If $b\in B$ and $b+1\notin B$, 
		this means $r_{b,n}\in\Z$ but $r_{b+1,n}\notin \Z$ and so
		$
			|r_{b+1,n}|_3 > 0 $ (since multiplying by $2$ keeps
		us in $\Z$, the only way to leave $\Z$ is to divide by $3$).
		However, $|r_{b+1,n}|_3 > 0$ implies that $r_{b+i,n}\notin \Z$
		for all $i>0$.  We conclude that $B$ cannot contain any gaps.

		Since $B$ consists of a contiguous set
		of integers, it will complete the proof if we 
		show that there do not exist two adjacent rows
		in $\Phi(y)$ where one requires $\Rf$ and the other requires $\Rc$.
		We will conclude the proof by showing that the rules of the
		Kari-Culik tiling forbid such an occurrence.

		Suppose $s=\Rf(\alpha,t)\neq\Rc(\alpha,t)=s'$ and
		$\alpha\notin\Q$, and notice $s$ and $s'$ differ only by a transposition
		of two adjacent coordinates.  For simplicity, assume
		$s$ and $s'$ differ at coordinates $1$ and $2$
		and that $\alpha(s)\in[0,1]$ so that $s$ and $s'$ consist
		of the symbols $0$ and $1$. We then have
		\[
			s=\cdots s_{0}s_1s_2s_3 \cdots\qquad\text{and}\qquad
			s'=\cdots s_{0}s_2s_1s_3 \cdots,
		\]
		and in particular $s_1\neq s_2$.  Since $s$ and $s'$ are
		both valid Sturmian sequences, we may conclude that $s_{0}=s_3$,
		since if $s_{0}\neq s_3$ either $s$ or $s'$ would contain both
		a $1,1$ and a $0,0$
		(which is impossible in a Sturmian sequence \cite{fogg}).  Further, since $s$ requires $\floor{\cdot}$,
		we know $s_1 > s_2$.

		In general, we will call a length four word $w_{\alpha,t}=w_0 w_1 w_2 w_3$
		or $w_{\alpha,t}=w_0w_2w_1w_3$
		a \emph{straddle word} of a Sturmian sequence if 
		\[
			w_0 w_1 w_2 w_3=(\Rf(\alpha,t))_{i}^{i+3}
			\qquad \text{and}\qquad
			w_0 w_2 w_1 w_3=(\Rc(\alpha,t))_{i}^{i+3}
		\]
		for some $i$ (or vice versa) and $w_1\neq w_2$.
		The previous argument shows that if $w_{\alpha,t}$ is a straddle
		word, then $w_0=w_3$.  It also shows
		that if a Sturmian sequence requires $\Rf$ or $\Rc$, it necessarily
		contains a straddle word.

		Now, consider a row $r$ of $y$ where the sequence of bottom labels requires
		$\Rf$ and the top labels require $\Rc$ (or vice versa) and let $w^t$
		and $w^b$ be the straddle words for the labels on the top of $r$
		and the bottom of $r$ respectively.  Since the top sequence requires $\Rf$
		and the bottom sequence requires $\Rc$, we conclude that $w^t_1 > w^t_2$
		and $w^b_1 < w^b_2$ (or vice versa if the roles of $\floor{\cdot}$
		and $\ceil{\cdot}$ are reversed).  We will call a pair of
		straddle words like these, whose middle two symbols satisfy opposite
		inequalities, \emph{misaligned}.

		We will complete the proof by observing that
		misaligned straddle words cannot occur in $y$.

		By enumerating all pairs of length-four words $(w^t,w^b)$
		that arise as tops and corresponding bottoms of
		rows of type $2.1$, we see that out of the $64$
		possibilities, only four have the property that
		$w^t_1\neq w^t_2$ and $w^b_1\neq w^b_2$ (which is necessary
		to be a straddle word). Out of those four, none are misaligned.
		Similarly, in a row of type $\frac{1}{3}$, of the $96$ possibilities,
		$24$ differ in their middle symbols and out of those, none are misaligned
		straddle words.

		Now consider a row of type 2.2.  Out of the $128$ possible sequences of
		length $4$, 
		there are exactly two ways to obtain misaligned straddle words, namely:
		\begin{center}
			\includegraphics[width=5in]{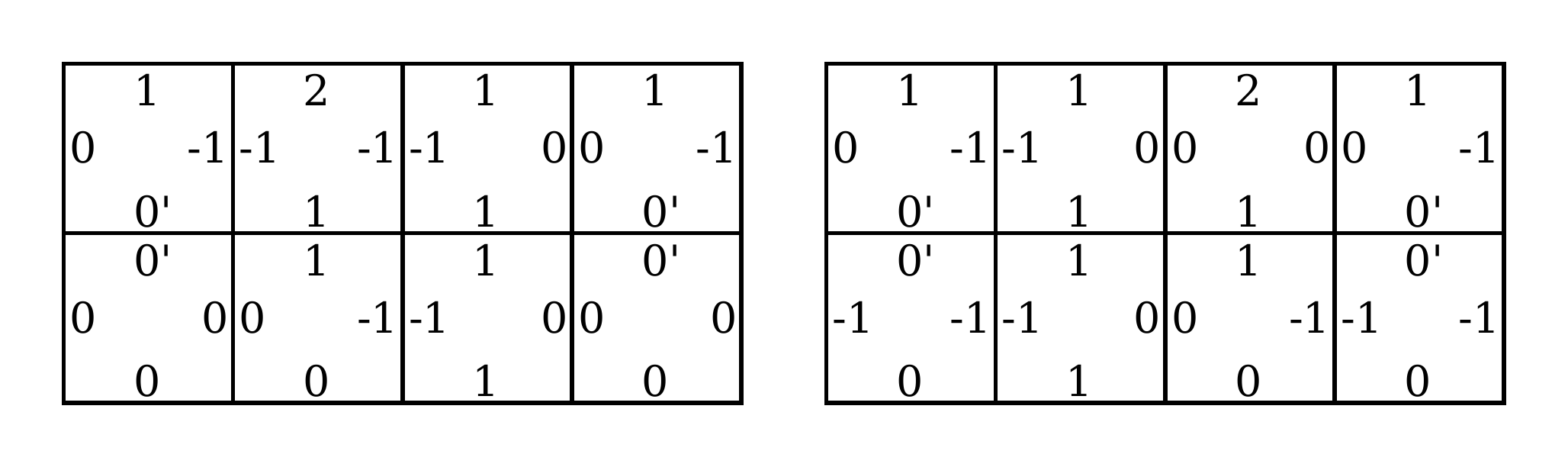}
		\end{center}
		This gives misaligned straddle pairs of
		$(w^t,w^b)=(1211,0010)$ and $(w^t,w^b)=(1121,0100)$.
		Since we are considering a type 2.2 row, the Sturmian angle
		for the sequence of bottoms must be in $[1/3,1/2]$.  Thus, there
		cannot be three $0$'s in a row.
		We therefore conclude the symbol before
		the word
		$w^b=0010$ must be a $1$ and the symbol
		after the word $w^b=0100$ must be a $1$. 
		Since we are considering $0010$ and $0100$ as straddle
		subwords of some pair of Sturmian sequences and these Sturmian
		sequences must agree everywhere except for a single transposition
		of symbols, we conclude $w^b=0010$ and $w^b=0100$ 
		must be subwords of $100101$ 
		and $101001$.  By a similar argument, 
		the top straddle words must be subwords
		of $211212$ and $212112$.  Thus, the 
		stacked tile to the left
		or right of the designated blocks must have a bottom label
		of $1$ and a top label of $2$.
		Inspecting the two stacked tiles with this property reveals that neither
		of them are compatible with the potential misaligned
		straddle words shown, and thus misaligned straddle words cannot
		appear in $y$.
	\end{proof}

\section{The Subset $KC$}
	Recall that $KC$ is the subset of Kari-Culik tilings whose bottom
	labels form generalized Sturmian sequences.  The following theorem allows
	us to focus on the Sturmian sequences made from 
	bottom labels of $KC$ as opposed to configurations on $\mathfrak K^{\Z^2}$.

	\begin{theorem}\label{PropPhiOneToOne}
		$\Phi|_{KC_{\Q^c}}$ is one-to-one and $\Phi$ is at most sixteen-to-one.
	\end{theorem}
	\begin{proof}
		Fix $x\in KC$ and consider a row $r$ of $x$. Let $r^t$ be the Sturmian
		sequence formed by the top labels of $r$ and let $r^b$ be the Sturmian
		sequence formed by the bottom labels of $r$.  Further, let 
		$\alpha^t=\alpha(r^t)$ and $\alpha^b=\alpha(r^b)$, and if $w^t=(r^t)_i^j$
		is a subword of $r^t$, then $w^b=(r^b)_i^j$ is the corresponding
		subword of $r^b$. That is $w^b$ is the subword of $r^b$ whose indices
		are identical to the indices of $w^t$. Note that give a $w^t$ absent of
		the indices from which it came, $w^b$
		may not be uniquely defined.

		Notice that by the multiplier property (Equation \eqref{EqMultProperty}), 
		$r$ is uniquely determined by $(r^t,r^b)$ and a single left label
		of one of the tiles in $r$.

		Define $\Phi_2$ by $\Phi_2(r)=(r^t,r^b)$ and extend $\Phi_2$
		to work on subwords of $r$.  By our previous observation, if there
		is a subword $r'\subset r$ such that $\Phi_2^{-1}(\Phi_2(r'))$ contains
		a single element, then $r$ is uniquely determined by $(r^t,r^b)$.
		We will show that if $x\in KC_{\Q^c}$, this is always the case.  Moving
		forward, we analyze $r$ separately depending on its type.

		Case $r$ is of type $\frac{1}{3}$:
		In this case, we know $\alpha^t\in[1/3,2/3]$ and $\alpha^b\in[1,2]$.

		\begin{figure}[h!]
			\footnotesize
			\begin{center}
			\includegraphics[height=2in]{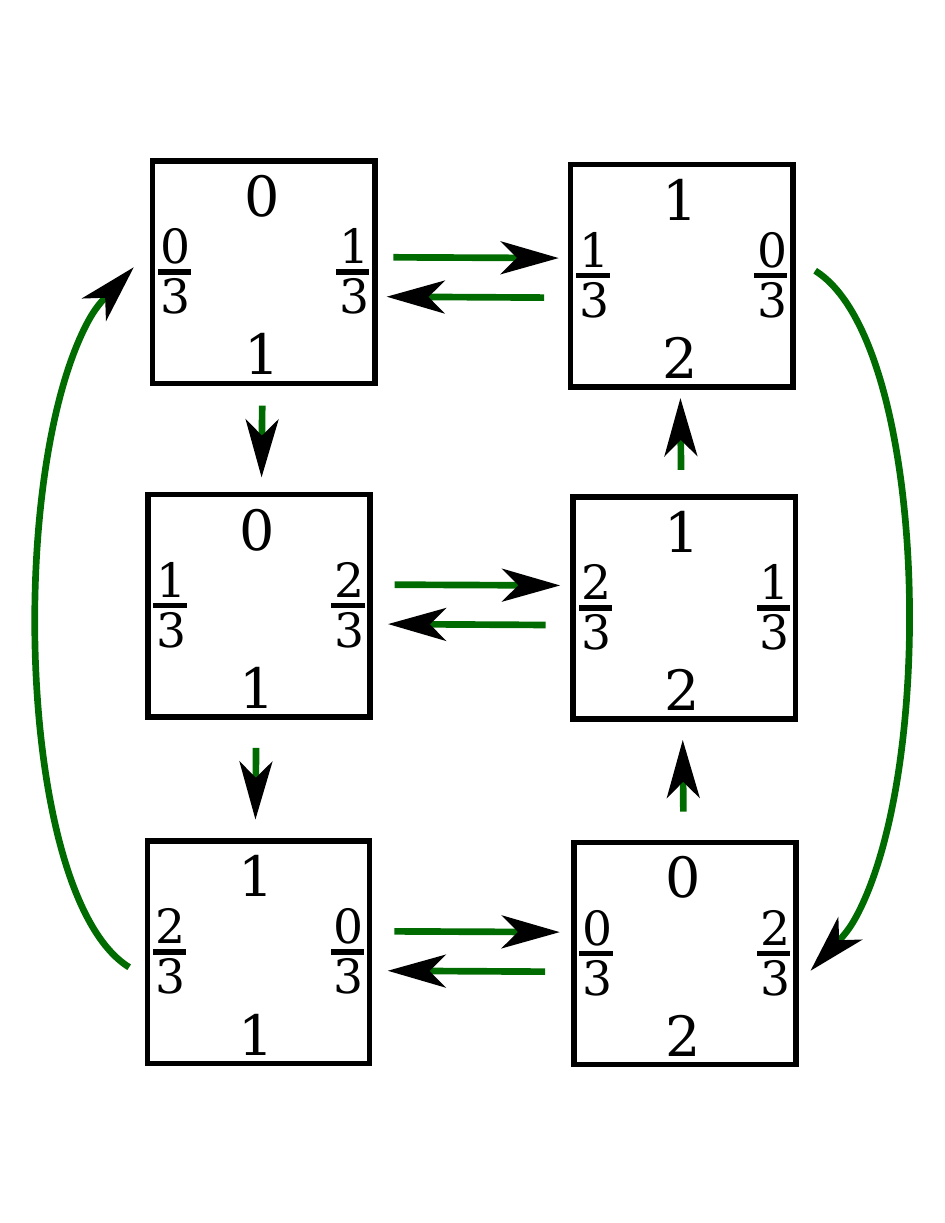}
			\end{center}
			\vspace{-3em}
			\caption{ \footnotesize Transition graph for type $\frac{1}{3}$
			tiles.}
			\label{FigType13TransitionGraph}
		\end{figure}

		Figure \ref{FigType13TransitionGraph} shows the transition graph moving
		left to right in a type $\frac{1}{3}$ row.  
		Notice that there is only one way for $11$ or $00$ to appear
		as top labels in a type $\frac{1}{3}$ row. In particular,
		if $w^t=11$ then $w^b=22$ and if $w^t=00$ then $w^b=11$
		and $|\Phi_2^{-1}(11,22)|=
		|\Phi_2^{-1}(00,11)|=1$.  Thus,
		if $r^t$ contains the word $11$ or $00$, $r$ is uniquely determined
		by $(r^t,r^b)$.
		
		If $r^t$ contains neither $11$ nor $00$, then $\alpha^t=1/2$ and
		$r^t=\cdots 101010\cdots $.  Analysing further, 
		if $w^t=01$ then $w^b\in\{12,21,11,22\}$.  We note that
		$|\Phi_2^{-1}(01,11)|=|\Phi_2^{-1}(01,22)|=|\Phi_2^{-1}(01,21)|=1$ and
		$|\Phi_2^{-1}(01,12)|=2$.  Thus, $\Phi$ on a row of type $\frac{1}{3}$
		is only non-unique if $\alpha^t=1/2$ which further implies $\alpha^b=3/2$.

		Case $r$ is of type $2.1$:
		In this case, we know $\alpha^t\in[1,2]$ and $\alpha^b\in[1/2,1]$.
		
		\begin{figure}[h!]
			\footnotesize
			\begin{center}
			\includegraphics[height=2in]{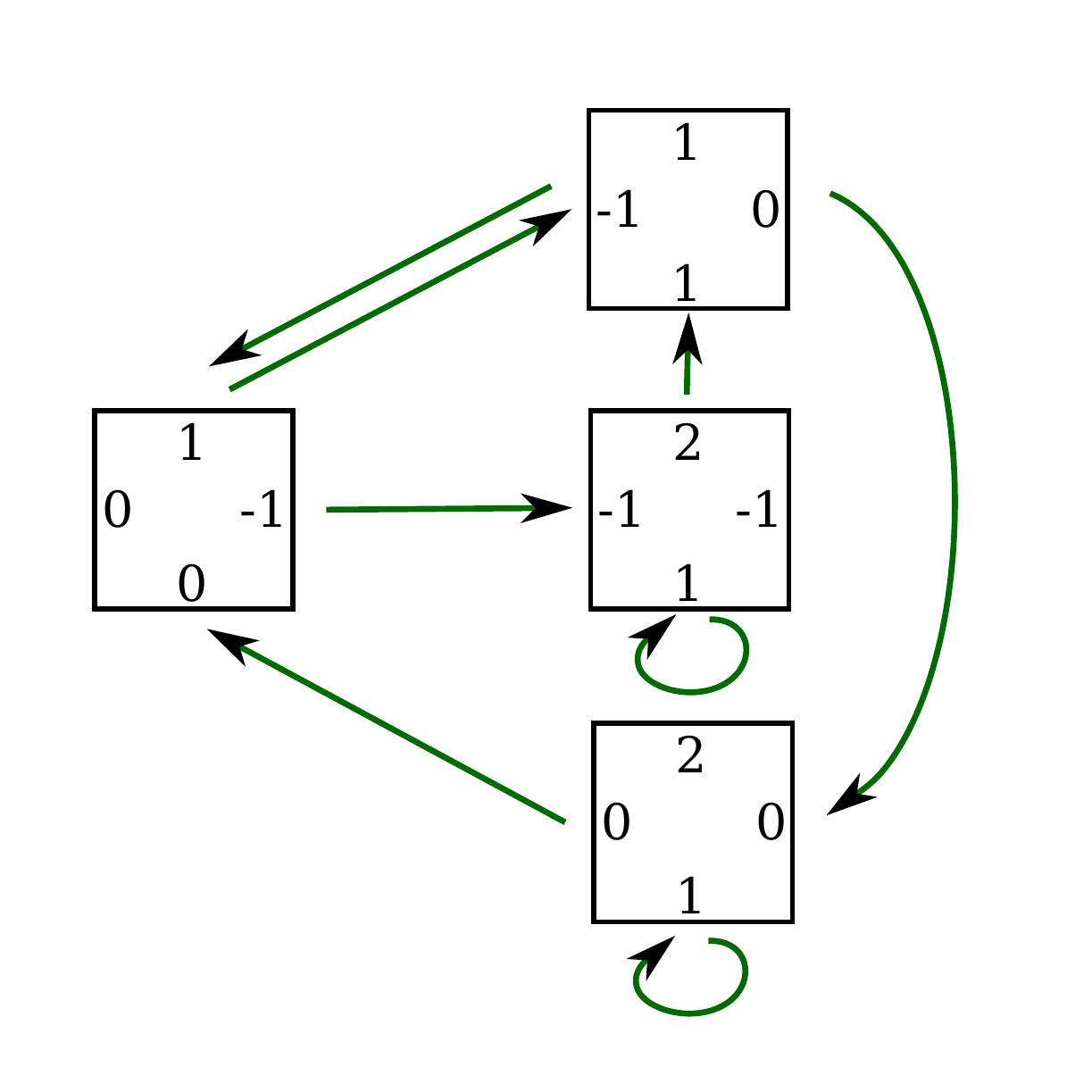}
			\end{center}
			\vspace{-3em}
			\caption{ \footnotesize Transition graph for a type $2.1$
			row.}
			\label{FigType21TransitionGraph}
		\end{figure}
		
		Figure \ref{FigType21TransitionGraph} shows the transition graph moving
		left to right in a type $2.1$ row.
		Notice that there is only one way
		a type $2.1$ row can contain $0$ as a bottom label.
		This means $r$ is uniquely determined by $(r^t,r^b)$ unless $\alpha^b=1$
		(since $\alpha^b<1$ implies a zero occurs in $r^b$) and consequently
		$r^b=\cdots 111\cdots$.
		If $w^b=1$ then $w^t\in\{1,2\}$ with
		$|\Phi_2^{-1}(1,1)|=1$ and
		$|\Phi_2^{-1}(2,1)|=2$.

		Case $r$ is of type $2.2$:
		In this case, we know $\alpha^t\in[4/3,2]$ and $\alpha^b\in[1/3,1/2]$
		(since we interpret our multiplier as $4$ in this case).

		Because $\alpha^b\in [1/3,1/2]$, we know that $r^b$ cannot
		contain $11$ (since if a rotation sequence contains $11$,
		then its angle must be strictly larger than $1/2$).  
		Therefore, it must contain 
		$100$ or $10101$ as a subword.
		Further, if $\alpha^b\in(1/3,1/2)$, $r^b$
		must contain $10100$ as a subword.  Let $w^b\in\{100,10101,10100\}$.
		Below is a list of all pairs $(w^t,w^b)$ such that
		$|\Phi_2^{-1}(w^t,w^b)|>1$.
		\begin{itemize}
			\item[] $(w^t,w^b)=(211,100)$,
			which can be obtained in exactly two ways, namely
			\vspace{-1em}
			\begin{center}
				\includegraphics[width=5in]{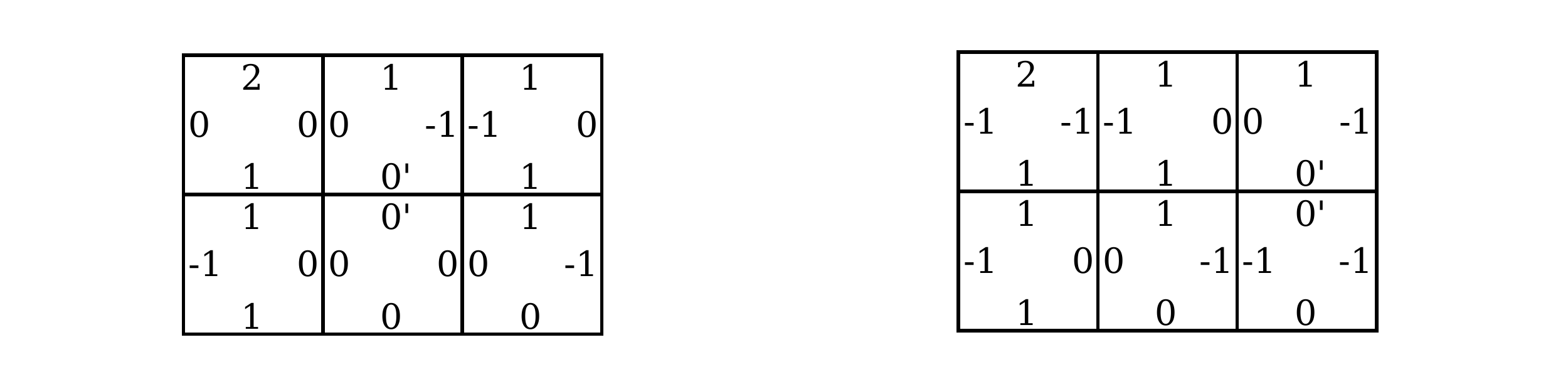}
			\end{center}
			\item[] $(w^t,w^b)=(22222,10101)$,
			which can be obtained in exactly two ways, namely
			\vspace{-1em}
			\begin{center}
				\includegraphics[width=5in]{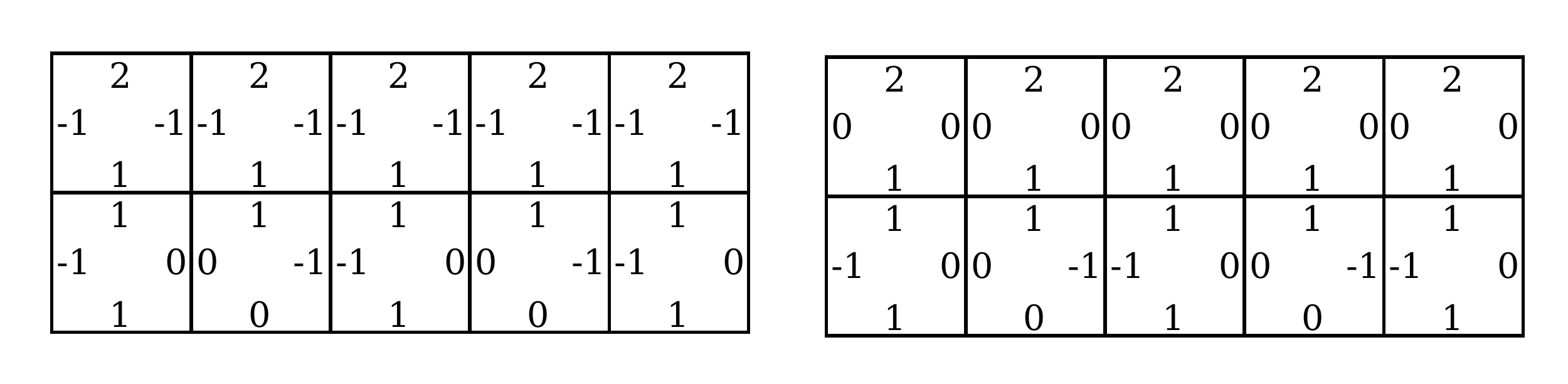}
			\end{center}
			\item[] and $(w^t,w^b)=(22211,10100)$,
			which can be obtained in exactly two ways, namely
			\vspace{-1em}
			\begin{center}
				\includegraphics[width=5in]{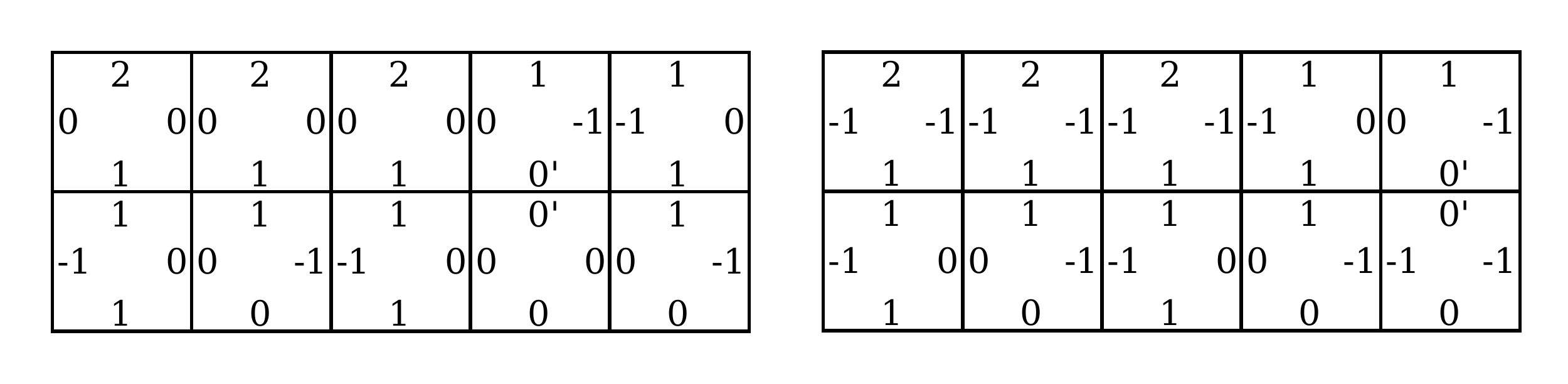}
			\end{center}
		\end{itemize}

		Considering the pair $(w^t,w^b)=(22211,10100)$, we see that
		$22211$ is not a generalized Sturmian sequence (since it
		contains both $11$ and $22$ as subwords), and so this situation
		never occurs.  This means if $\Phi_2$ is not invertible,
		$\alpha^b=1/3$, which corresponds to the
		first case, or $\alpha^b=1/2$ which corresponds to the second case.

		As a result of our case-by-case analysis, we see that if 
		$\alpha^b\notin\{1/3,1/2,1,3/2\}$, then $r$ is uniquely determined
		and so when restricted to $KC_{\Q^c}$, $\Phi$ is one-to-one.
		Further, since for $\alpha^b\in\{1/3,1/2,1,3/2\}$ we have
		$\Phi_2$ is at most two-to-one, we conclude in general that
		$\Phi$ is at most sixteen-to-one (since $f$ prevents $\alpha^b$
		from taking any particular value in $\{1/3,1/2,1,3/2\}$ more than once).
	\end{proof}

	In light of Theorem \ref{PropPhiOneToOne}, we will treat points in $KC$
	and points in $\Phi(KC)$ interchangeably, differentiating only when needed.

	Recall that $T:KC\to KC$ is the horizontal shift and $S:KC\to KC$ is the vertical
	shift.

	\begin{proposition}
		For $x\in KC$, we have
		\[
			\alpha(Sx) = f(\alpha(x)),
		\]
		where $f$ is applied component-wise and if $x\in KC_{\Q^c}$,
		\[
			t(Tx) = t(x) + \alpha(x).
		\]
	\end{proposition}
	\begin{proof}
		This immediately follows from Corollary \ref{PropAnglesRespect} and
		the definition of a rotation sequence.
	\end{proof}

	We will now produce a parameterization of points in $KC_{\Q^c}$
	in a similar fashion to the way a Sturmian sequence may be parameterized
	by an angle, phase, and choice of floor or ceiling function.

	\begin{definition}
		Let \[
			\T = \varprojlim \R/(6^n\Z)
		\]
		be the inverse limit of the groups $\R/(6^n\Z)$ as $n\to\infty$.
	\end{definition}

	We view a point $t=(t_0,t_1,\ldots)\in \T$ as a sequence of real numbers
	endowed with the product topology and
	satisfying the consistency condition $t_i = t_{i+1}\mod 6^i$ with
	$\proj_k(t) = t_k$.  As such,
	we can define scalar-multiplication functions $M_a$ and $M_{1/a}$ for $a\in\{2,3\}$
	as follows:
	\[
		M_a(t_0,t_1,\ldots) = (at_0\mod 1,at_1\mod 6,at_2\mod 6^2,\ldots)
	\]
	and
	\[
		M_{1/a}(t_0,t_1,\ldots) = M_{6/a}(t_1/6,t_2/6,t_3/6,\ldots).
	\]

	\begin{proposition}\label{PropMWellDefined}
		$M_2,M_3,M_{1/2},M_{1/3}:\T\to\T$ are bijective homomorphisms.
	\end{proposition}

	The proof of Proposition 
	\ref{PropMWellDefined} is straightforward, relying on the fact
	that $2$ and $3$ divide $6$.  From now on, for $t\in \T$ we 
	may write $at$ or $t/a$ instead 
	of $M_at$ and $M_{1/a}t$.  Further, for $r\in \R$,
	we may define scalar addition $A_r:\T\to\T$
	by
	\[
		A_r(t_0,t_1,\ldots) = (t_0+r\mod 1,t_1+r\mod 6,t_2+r\mod 6^2,\ldots)
	\]
	and we may write $t+r$ instead of $A_r(t)$.

	\begin{notation}
		Extend $f$ to a function $\hat f:[1/3,2]\times\T\to [1/3,2]\times\T$ by
		\[
			\hat f(\alpha,t) = 
				\left\{\begin{array}{cl}
						(2\alpha,2t) &\text{ if } \alpha \in[1/3,1)\\
						(\alpha/3,t/3) &\text{ if } \alpha\in [1,2]
				\end{array}\right. .
		\]
	\end{notation}

	Notice that $\hat f:[1/3,2]\times \T\to[1/3,2]\times \T$ is a bijection.
	Morally, we will show that $[1/3,2]\times \T$ is a parameterization
	of $KC$.  However, since trouble arises for generalized Sturmian sequences
	with rational angles and Sturmian sequences whose phase vector
	contains zero, we focus our attention to the sets
	$([1/3,2]\backslash \Q)\times\T\times\{\Rf,\Rc\}$ 
	and $KC_{\Q^c}$.


	\begin{lemma}\label{PropWExists}
		Let $\Orb_1$ be the orbit of $1$ under $f$, and let 
		\[X=\{(\vec \alpha, \vec t): (\vec \alpha, \vec t)\text{
		satisfies the BC property and }1\notin\vec \alpha\}.\]
		There exists a bijection $W:([1/3,2]\backslash\Orb_1)\times\T\to X$ that respects
		the dynamical relationships.  That is, for $(\alpha,t)\in([1/3,2]\backslash\Orb_1)\times \T$ such
		that $W(\alpha, t)=(\vec \alpha, \vec t)$, we have
		\[
			W(\alpha, t+\alpha) = (\vec \alpha, \vec t+\vec \alpha\mod 1)\qquad\text{and}\qquad
			W\circ \hat f(\alpha, t)=(\sigma(\vec \alpha), \sigma(\vec t)).
		\]
	\end{lemma}
	\begin{proof}
		Defining $W$ is straightforward.  Fix $(\alpha,t)\in([1/3,2]\backslash\Orb_1)
		\times \T$ and define $W(\alpha,t) = (\vec \alpha, \vec t)$ where
		$(\alpha_i,t_i) = 
		(\id\times \proj_0)\circ \hat f^i(\alpha,t)$.
		The definition of $\hat f$ acting on $([1/3,2]\backslash\Orb_1)\times\T$
		ensures that $(\vec \alpha, \vec t)$ satisfies the BC property
		and respects the desired dynamical relationships.  
		
		$W$ is clearly one-to-one in the first coordinate.  Further, if
		$t,t'\in \T$ with $t\neq t'$, we have that $\proj_j t\neq\proj_j t'$
		for some $j$.  From this we deduce that $\proj_0\circ \hat f^{j'}(t)\neq \proj_0\circ \hat f^{j'}(t')$
		for some $j'\geq j$,
		and so $W$ is one-to-one.
		After the construction of $W^{-1}$,
		it will be evident that it is onto.

		Constructing $W^{-1}$ is slightly more difficult.
		Fix $(\vec \alpha, \vec t)\in
		X$.  Let $\Lambda_i = \alpha_i/\alpha_0$.  Let $|\cdot|_2$
		and $|\cdot|_3$ be the 2-valuation and 3-valuation respectively,
		and note that $\Lambda_i$ is always rational and so its $2$ and $3$ valuations
		are defined.

		Let $j(i) = \min\{n\geq 0: |\Lambda_n|_3 = i\}$ and 
		define the subsequences $(t^{(3)}_i)_{i=0}^\infty$ 
		and $(\Lambda^{(3)}_i)_{i=0}^\infty$ of $(t_i)_{i=0}^\infty$
		and $(\Lambda_i)_{i=0}^\infty$
		by 
		\[
			t^{(3)}_i = t_{j(i)}\qquad\text{and}\qquad\Lambda^{(3)}_i=\Lambda_{j(i)}.
		\]
		Similarly, let $j'(i) = \min\{n\geq 0:|\Lambda_{-n}|_2 = i\}$
		and define the subsequences $(t^{(2)}_i)_{i=0}^\infty$
		and  $(\Lambda^{(2)}_i)_{i=0}^\infty$
		by $t^{(2)}_i = t_{-j'(i)}$ 
		and $\Lambda^{(2)}_i=\Lambda_{-j'(i)}$.  We've constructed
		$t^{(2)}$ and $t^{(3)}$ to be the values of $\vec t$ corresponding to where 
		we ``divide $\alpha$ by 2'' or ``divide $\alpha$ by 3''
		respectively.

		To construct $W^{-1}(\vec \alpha,\vec t)$, first consider the point 
		$z^{(3)}=(z^{(3)}_0,z^{(3)}_1,\ldots)\in \varprojlim \R/(3^n\Z)$
		defined inductively in the following way.
		Let $z^{(3)}_0 = t^{(3)}_0=t^{(2)}_0$.  Fix $j$ and suppose for all $i<j$ we
		have
		\begin{align*}
			\Lambda^{(3)}_i z_i^{(3)} & = t^{(3)}_i\mod 1.
		\end{align*}
		Let $p$ be such that $\Lambda^{(3)}_j=\frac{2^p}{3}\Lambda^{(3)}_{j-1}$.
		We now have that $3t^{(3)}_j=2^pt^{(3)}_{j-1}\mod 1$ and so along
		with our induction hypothesis there exist unique $r',r''\in\{0,1,2\}$ so
		\begin{align*}
			3t^{(3)}_j = 2^pt^{(3)}_{j-1}+r' &\mod 3\\
			\Lambda^{(3)}_{j-1}z^{(3)}_{j-1}-t^{(3)}_{j-1} =r'' &\mod 3.
		\end{align*}
		Define $z^{(3)}_j=z^{(3)}_{j-1}+r3^{j-1}$ where $r\in\{0,1,2\}$ 
		is the unique solution to $2^p r'' - r' + r2^p\Lambda^{(3)}_{j-1}3^{j-1}=0\mod 3$.
		Note that since $2^p\Lambda^{(3)}_{j-1}3^{j-1}\in\Z$ and contains no
		multiples of $3$, such an $r$ always exists.
		By construction, $z^{(3)}_{j}$ satisfies $z^{(3)}_{j-1}=z^{(3)}_j\mod 3^{j-1}$.
		We will now verify that $\Lambda^{(3)}_jz^{(3)}_j-t^{(3)}_j=0\mod 1$.
		
		Substituting, we see 
		\[
			\Lambda^{(3)}_jz^{(3)}_j-t^{(3)}_j = 
			\tfrac{2^p}{3}\Lambda^{(3)}_{j-1}z^{(3)}_j-t^{(3)}_j =	
			\tfrac{2^p}{3}\Lambda^{(3)}_{j-1}(z^{(3)}_{j-1}+r3^{j-1})-t^{(3)}_j.
		\]
		Finally, multiplying by $3$, we have	
		\[
			2^p\Lambda^{(3)}_{j-1}(z^{(3)}_{j-1}+r3^{j-1})-3t^{(3)}_j
			=2^p\Lambda^{(3)}_{j-1}(z^{(3)}_{j-1}+r3^{j-1})-2^pt^{(3)}_{j-1}-r'
			\mod 3
		\]  \[
			= 2^p\Big(
			\Lambda^{(3)}_{j-1}z^{(3)}_{j-1}-t^{(3)}_{j-1}+r\Lambda^{(3)}_{j-1}3^{j-1}
			\Big)-r' = 
			2^p r'' - r' + r2^p\Lambda^{(3)}_{j-1}3^{j-1}=0\mod 3
		\]
		as desired.  

		We have shown that $z^{(3)}$ exists and is unique. In an analogous way,
		construct $z^{(2)}\in\varprojlim\R/(2^n\Z)$.
		Finally,
		since $z^{(3)}_0=z^{(2)}_0$, by the Chinese remainder theorem
		we may produce $z=(z_0,z_1,\ldots)\in \T$ 
		such that $z^{(3)}_i=z_i\mod 3^i$ and $z_i^{(2)}=z_i\mod 2^i$.

		It is worth noting now that by construction, $z_i$ is the unique
		simultaneous solution to
		\begin{align*}
			\Lambda^{(3)}_i z_i & = t^{(3)}_i\mod 1\\
			\Lambda^{(2)}_i z_i & = t^{(2)}_i\mod 1
		\end{align*}
		where $z_i\in\R/(6^i\Z)$ and $z_{i-1} = z_{i}\mod 6^{i-1}$.  

		Having established existence and uniqueness of $z$,
		we may define $W^{-1}(\vec\alpha,\vec t)
		= (\alpha_0, z)$.  The fact that $W^{-1}$ is an inverse is
		now immediate by construction:
		if $|\Lambda_j|_3 = k\geq 0$, then $\Lambda_j\proj_k(z) = t_j^{(3)}\mod 1$,
		and if 
		$|\Lambda_j|_2 = k\geq 0$, then $\Lambda_j\proj_k(z) = t_{-j}^{(2)}\mod 1$.
	\end{proof}
	
	We restricted the domain of $W$ to $([1/3,2]\backslash\Orb_1)\times \T$ instead
	of $[1/3,2]\times \T$ because the function $f$ could be defined so that $f(1)=1/3$
	or $f(1)=2$, and both ways would be consistent with the rules of the Kari-Culik tilings.
	As such, this presents a small obstruction to $W$ being a bijection on $[1/3,2]\times \T$.

	To allow for a clearer statement of Proposition \ref{PropEquivParameters},
	we will extend notation so that $\hat f(\alpha, t,R)=(\alpha',t',R)$
	where $(\alpha',t')=\hat f(\alpha, t)$.

	\begin{proposition}\label{PropEquivParameters}
		There exists maps $K:[1/3,2]\times \T\times\{\Rf,\Rc\}\to KC$
		and $K':KC_{\Q^c}\to[1/3,2]\times \T\times\{\Rf,\Rc\}$ so that $K\circ K'=\id$
		and for any $y\in KC_{\Q^c}$,
		$K'$ satisfies the following relationships:
		\begin{align*}
			K'(y)&=(\alpha,t, R)\\
			K'(Ty)&=(\alpha,t+\alpha, R)\\
			K'(Sy)&=\hat f(\alpha,t, R)
		\end{align*}
		for some $(\alpha,t,R)\in[1/3,2]\times\mathcal T\times\{\Rf,\Rc\}$.
	\end{proposition}
	\begin{proof}
		This proposition is a corollary of Lemma \ref{PropWExists}.

		Let $X=\{(\vec \alpha, \vec t): (\vec \alpha, \vec t)\text{
		satisfies the BC property}\}$.
		Observe that by Proposition \ref{PropBasicConstructRepr},
		we have explicit maps $A:X\times\{\Rf,\Rc\}\to KC$ and
		$A':KC_{\Q^c}\to X\times \{\Rf,\Rc\}$ so that $A\circ A'=\id$.
		Projecting onto the first coordinate,
		 we also have if $A'(x) = (\vec\alpha,\vec t)$,
		then $A'(Tx) = (\vec\alpha, \vec t+\vec \alpha\mod \vec 1)$
		and $A'(Sx) = (\sigma(\vec \alpha), \sigma(\vec t))$, where $\sigma$ is 
		the shift on bi-infinite vectors.

		Lemma \ref{PropWExists}
		gives us 
		an invertible map
		$W:([1/3,2]\backslash\Orb_1)\times \T\to X$ that respects the dynamical relationships.
		Without changing notation, extend $W$ (non-bijectively) to a map $W:[1/3,2]\times \T\to X$
		in a way that respects the dynamical relationships.  We may now define $K=A\circ W$.
		Finally, since the first component of the image of $A'$ avoids $\Orb_1$, we may define
		$K'=W^{-1}\circ A'$.
	\end{proof}

	Where convenient, we will think of $K:[1/3,2]\times\T\to KC$ and
	$K':KC_{\Q^c}\to [1/3,2]\times \T$ without the notational burden
	of keeping track of
	the choice of $\Rf$ or $\Rc$.  Further, since points in $KC_{\Q^c}$ requiring
	$\Rc$ to be expressed may be approximated by points in $KC_{\Q^c}$ 
	requiring only $\Rf$ (and vice versa), the choice of $\Rf$ and $\Rc$ may be
	safely ignored.
	Considering the relationships outlined in the proof of
	Proposition \ref{PropEquivParameters}, we may now deduce the following
	theorem.

	\begin{theorem*}[Theorem \ref{ThmKCConjSkew}]
		$KC_{\Q^c}$ is a skew product.  That is, if $K'$ is 
		the map defined in Proposition \ref{PropEquivParameters}
		and $y\in KC_{\Q^c}$, we have the following relationships:
		\begin{align*}
			K'(y)&=(\alpha,t)\\
			K'(Ty)&=(\alpha,t+\alpha)\\
			K'(Sy)&=\hat f(\alpha,t).
		\end{align*}
	\end{theorem*}

\section{Minimality of $KC$}
	In light of Theorem \ref{ThmKCConjSkew}, we will first consider
	the minimality of the $\Z^2$ action of 
	$(\hat T, \hat S)$ on $[1/3,2]\times \T$ where $\hat T$ and 
	$\hat S$ are defined by
	\[
		\hat T(\alpha, t) = (\alpha, t + \alpha)\qquad\text{and}
		\qquad
		\hat S(\alpha, t) = \hat f(\alpha, t).
	\]
	Though the minimality of $([1/3,2]\times \T,\hat T, \hat S)$ and $(KC, T, S)$
	can be deduced
	from the explicit methods in Section \ref{SecExplicitBounds}, we present an abstract proof 
	of minimality that may be of independent interest.

	Since trouble arises when considering $(\alpha, t)$ with $\alpha\in\Q$,
	we will use the Uryson Metrization Theorem to produce a metric $\hat d$
	such that 
	$([1/3,2]\backslash \Q)\times \T$, endowed with the subspace topology, is a complete 
	metric space with respect to $\hat d$.

	\begin{proposition}\label{PropParamSpaceMinimal}
		The $\Z^2$ action of $(\hat T, \hat S)$ on 
		$([1/3,2]\backslash \Q) \times\T$ is minimal with respect to $\hat d$.
	\end{proposition}
	\begin{proof}
		We first claim that since $\alpha\notin\Q$, the second coordinate
		of $\Orb_{\hat T}(\alpha, t)$ is dense in $\T$.
		For any $k$, it is clear that the second coordinate of 
		$\Orb_{\hat T}(\alpha, t)$ is dense modulo $6^k$.
		That is, the second coordinate of $\id\times\proj_k(\Orb_{\hat T}(\alpha, t))$
		is dense 
		in $\proj_k(\T)$.  We therefore have the second coordinate of 
		$\Orb_{\hat T}(\alpha, t)$ is dense modulo $6^i$ for any $i\leq k$.
		Denseness now follows from 
		the definition of the product topology on $\T$.

		Using this observation, it is clear that for any
		$(\alpha, t'),(\alpha, t)\in ([1/3,2]\backslash \Q)\times\T$,
		we have $(\alpha, t) \in \overline{\Orb_{\hat T}(\alpha, t')}$.

		To complete the proof,
		fix $(\alpha', t'),(\alpha, t)\in ([1/3,2]\backslash \Q)\times\T$.
		Since the orbit of any point under $f:[1/3,2]\to[1/3,2]$ is dense,
		we may find a point $(\alpha, t'')\in\overline{\Orb_{\hat S}(\alpha',t')}$.
		Using our previous observation, 
		$(\alpha, t) \in \overline{\Orb_{\hat T}(\alpha, t'')}$. This
		means
		$(\alpha, t)\in \overline{\Orb(\alpha',t')}$, and so the orbit of every point
		is dense.
	\end{proof}

	If the function $K:([1/3,2]\backslash \Q)\times\T\to KC_{\Q^c}$
	were continuous, this would
	give us a quick proof of the minimality of $KC_{\Q^c}$.  
	Unfortunately this is not the case,
	but the set of points where $K$ is continuous is a dense $G_\delta$.
	
	\begin{proposition}\label{PropKCtsOnDenseSet}
		The set of points of continuity of $K:[1/3,2]\times\T\to KC$ is
		a dense residual set $G$ and $K(G)$ dense in $KC$.
	\end{proposition}
	\begin{proof}
		Let $A_{m\times n}= \{0,\ldots, m-1\}\times\{0,\ldots,n-1\}$.
		Suppose $E\subset [1/3,2]\times \T$ is an open set on which $K$ is not
		continuous.
		Then, for some $m,n\in \N$, $E$ must contain a point $(\alpha,t)$ so that 
		$K(\alpha,t,\Rf)|_{A_{m\times n}}\neq K(\alpha,t,\Rc)|_{A_{m\times n}}$.
		Any point $(\alpha,t)$ that satisfies this is a point of discontinuity, and
		any point of discontinuity satisfies this condition for some $m,n$.
		
		Define
		\[
			B_{m\times n} = \{(\alpha, t): K(\alpha, t,\Rf)|_{A_{m\times n}}\neq K(\alpha,t,\Rc)|_{A_{m\times n}}\},
		\]
		and notice that since the only points of discontinuity of
		$\floor{\cdot}$ and $\ceil{\cdot}$ are $\Z$, $B_{m\times n}$ is a closed,
		nowhere dense set implying that $B_{m\times n}^c$
		is a dense open set.  Let
		\[
			G=\bigcap _{m,n\in\N} B_{m\times n}^c.
		\]
		We now have that by construction, the set of continuity 
		points of $K$ is the dense residual set $G$.  Further,
		since $K(B_{m\times n}^c)|_{A_{m\times n}}$ contains every
		$m\times n$ configuration that appears in $KC$, $K(G)$ is dense in $KC$.
	\end{proof}

	\begin{corollary}\label{PropKSomewhatCont}
		If $O\subset KC$ is a non-empty open set, then $K^{-1}(O)$ contains
		a non-empty open set.
	\end{corollary}
	\begin{proof}
		Let $G$ be a dense set of points of continuity of $K$ whose image is dense.
		Fix an open set $O\subset KC$.  Since $K(G)$ is dense, $K(G)\cap O\neq \emptyset$
		and so $O$ is a neighborhood of the image of a point of continuity of $K$.
		Thus $K^{-1}(O)$ is a neighborhood and so contains an open set.
	\end{proof}
	
	\begin{lemma}\label{PropCtsOnDenseImpliesMinimal}
		Suppose $(X,\sigma)$ and $(Y,\hat \sigma)$ are dynamical
		systems and that $g:X\to Y$ is a surjective
		map that satisfies $g\circ \sigma=\hat\sigma\circ g$.
		If the set of points of continuity of $g$ is a dense set $G$ and $g(G)$ is dense in $Y$, then 
		$(X,\sigma)$ minimal implies $(Y,\hat \sigma)$ minimal.
	\end{lemma}
	\begin{proof}
		We will first show that if $D\subset X$ is dense, then $g(D)$ is dense.
		Fix $y\in Y$ and some neighborhood $N_y$ of $y$. Let $G$ be a dense set of 
		points of continuity for $g$ such that $g(G)$ is dense.  Then, 
		$N_y\cap g(G)\neq \emptyset$ and so $N_y$ is a neighborhood
		of the image of a point of continuity.  Thus $g^{-1}(N_y)$ is a neighborhood
		of some point.  Since $D$ is dense, $D\cap g^{-1}(N_y)\neq \emptyset$,
		and so $g(D)$ intersects every neighborhood and must be dense.
		
		To complete the proof, fix $y\in Y$, 
		and by surjectivity of $g$ find $x\in X$ so
		$g(x)=y$.
		Suppose $X$ is minimal.  We have that
		$\Orb x$ is dense, and so $g(\Orb x) = \Orb g(x)=\Orb y$ is dense.  
	\end{proof}

	\begin{proposition}
		$\Phi(KC)=\overline{\Phi(KC_{\Q^c})}$.
	\end{proposition}
	\begin{proof}
		Fix $y\in \Phi(KC)$, $A=\{0,1,\ldots, m-1\}\times\{0,1,\ldots, n-1\}$,
		and the cylinder set $C=\{x\in \Phi(KC): x|_A=y|_A\}$.  $C$ is open
		and so by Corollary \ref{PropKSomewhatCont}, $K^{-1}(C)$ contains an
		open set.  Thus, there is some $(\alpha,t)\in K^{-1}(C)$ where
		$\alpha\notin \Q$.
	\end{proof}

	\begin{theorem*}[Theorem \ref{ThmKCMiminal}]
		$KC$ is minimal with respect to the group action of $\Z^2$ by translation.
	\end{theorem*}
	\begin{proof}
		We will first show that every orbit of every point in $KC_{\Q^c}$
		is dense in $KC'$, the subset of $KC$ where $\Phi$ is one-to-one,
		and then that the
		orbit closure of any point in $KC$ intersects $KC_{\Q^c}$.
		Finally we will show that $KC'$ is dense in $KC$.
		Here we must take special care to differentiate between $KC$ and $\Phi(KC)$.
		
		By Proposition \ref{PropEquivParameters}, we have that
		\[
			K\Big(([1/3,2]\backslash\Q)\times \T\times\{\Rf,\Rc\}\Big) = KC_{\Q^c}.
		\]
		Now, by Proposition \ref{PropKCtsOnDenseSet}, $K$ satisfies the conditions
		of Lemma \ref{PropCtsOnDenseImpliesMinimal}.  Thus, since Proposition
		\ref{PropParamSpaceMinimal} states that $([1/3,2]\backslash \Q)\times\T$
		is minimal with respect to $(\hat T,\hat S)$,
		Lemma \ref{PropCtsOnDenseImpliesMinimal} gives us that the orbit of every point
		in $KC_{\Q^c}$ is dense in $KC_{\Q^c}$.

		Let 
		\[
			KC'=\{x\in KC: \alpha(x)\text{ does not
			contain }1/3, 1/2, 1,\text{ or } 3/2\}.
		\]
		and recall that the proof of 
		Theorem \ref{PropPhiOneToOne} shows that $\Phi$
		is one-to-one exactly on $KC'$.  Thus, since $\Phi(KC_{\Q^c})$
		is dense in $\Phi(KC)$, by continuity of $\Phi$, we may conclude that the orbit of any
		point in $KC_{\Q^c}$ is dense
		in $KC'$.

		Next, we will show that for any $y\in KC$, we have 
		$\overline{\Orb y}\cap KC_{\Q^c}\neq \emptyset$.
		Fix $y\in KC$ and let $\alpha(y)=(\ldots, \alpha_0,\alpha_1,\ldots)$.  
		Choose $f$ to be the function \[
			f(x) =  \left\{\begin{array}{cl}
					2x &\text{ if } x\in[1/3,1)\\
					x/3 &\text{ if } x\in [1,2]
				\end{array}\right.
		\]
		or
		\[
			f(x) =  \left\{\begin{array}{cl}
					2x &\text{ if } x\in[1/3,1]\\
					x/3 &\text{ if } x\in (1,2]
				\end{array}\right.
		\]
		such that $f(\alpha_i)=\alpha_{i+1}$.
		Since the orbit of
		every point under $f$ is dense, we may find $y'\in \overline{\Orb y}$
		so that $\alpha(y')$ contains only irrationals.  However, since $\alpha(y')$
		contains only irrationals, $y'\in KC_{\Q^c}$ and so $\Orb (y')$
		is dense in $KC_{\Q^c}$ showing that $\Orb y$ is dense in $KC_{\Q^c}$.
		
		To complete the proof, we will now show $KC'$
		is dense in $KC$.  We will do this by considering cases.
		Suppose $y\in KC\backslash KC'$. This means for
		some $j$, $\alpha((y)_j)\in\{1/3,1/2,1,3/2\}$.  For simplicity,
		assume this occurs at $j=0$ and let $\alpha=\alpha((y)_0)$.

		Case $\alpha=3/2$: If $\alpha=3/2$, the sequence
		of bottom labels must be $\cdots 121212\cdots$.
		Looking at the transition graph in
		Figure \ref{FigType13TransitionGraph},
		a sequence to bottom labels of $\cdots 121212\cdots$ can be realized in
		two ways.  Call the configuration
		using the tiles in the bottom of the diagram configuration $A$
		and the configuration using the tiles in the top of the diagram configuration
		$B$, and notice
		that if $\alpha = 3/2+\delta$ for $0<\delta$ small, then the row will
		contain arbitrarily long runs of tiles in configuration $A$.  Similarly,
		if $\alpha = 3/2-\delta$, the row will contain arbitrarily long rows of
		tiles in configuration $B$.

		Case $\alpha=1$. Looking at the transition graph in Figure
		\ref{FigType21TransitionGraph}, we see that if $\alpha=1$, a bottom sequence
		of $\cdots 111\cdots$ can be obtain in two different ways.  Call
		these configurations configuration $A$ and configuration $B$.  Notice
		that we can force the bottom labels of the row to contain arbitrarily
		long sequences of $1$'s separated by $0$'s by picking
		$\alpha = 1-\delta$ for some small $0<\delta$.
		Further, the only way this can happen is by alternating arbitrarily
		long occurrences of configuration $A$ with arbitrarily long sequences
		of configuration $B$.

		Similarly for cases $\alpha=1/2$ and $\alpha=1/3$, a small 
		perturbation of $\alpha$ will produce arbitrarily long occurrences
		of each type of configuration. Since $KC'$ contains all perturbations
		of angles of points in $KC$, $KC'$ is dense in $KC$.
	\end{proof}

\section{Explicit Return Time Bounds}
	\label{SecExplicitBounds}
	We will now give explicit bounds on the on the size
	of the smallest rectangular configuration in $KC$ that
	contains every $m\times n$ sub-configuration.  The strategy
	will be to analyze the parameter space $[1/3,2]\times \T$
	to find intervals of
	parameters that have short $\hat T$-return times and  then bound the 
	$\hat S$-return times to such intervals.  These return time bounds 
	will then carry forward
	to $(KC, T,S)$.

	\begin{definition}
		Let $\mathcal P_{m,n}$ be the partition of $[1/3,2]\times\T$ given
		by $m\times n$ configurations in $KC$.  Specifically,
		$(\alpha,t)\sim(\alpha',t')$ if  for
		$A= \{0,\ldots, m-1\}\times\{0,\ldots,n-1\}$ we have
		\[
			K(\alpha,t,\Rf)|_A = K(\alpha',t',\Rf)|_A.
		\]
	\end{definition}
	\begin{definition}
		For a partition $\mathcal P$ of $[1/3,2]\times X$, let $\pi_\alpha(\mathcal P)$
		be the restriction of $\mathcal P$ to the fiber $\{\alpha\}\times X$.
	\end{definition}

	We will familiarize ourselves 
	with the structure of $\mathcal P_{m,n}$.  Let
	us consider $\mathcal P_{1,n}$.  Putting the inverse-limit space $\T$
	aside for a moment, let $\mathcal P_n$ be the partition
	of $[1/3,2]\times[0,1)$ such that $(\alpha,t)\sim(\alpha',t')$
	if $(\Rf(\alpha,t))_0^{n-1}=(\Rf(\alpha',t'))_0^{n-1}$.

	After considering pre-images
	under rotation by an angle $\alpha$,
	we see that $\pi_\alpha(\mathcal P_n)$ is precisely the partition
	generated by intervals whose
	endpoints are consecutive elements of
	$C_\alpha=\{0,-\alpha,-2\alpha,\ldots,-n\alpha\mod 1\}$.
	We view $C_\alpha$ as the places $[0,1)$ needs to be ``cut'' to produce 
	$\pi_\alpha(\mathcal P_n)$.  Now, varying $\alpha$, we see that $\mathcal P_n$ is produced
	by cutting $[1/3,2]\times[0,1)$ by the set of lines 
	$L=\{(x,y)\in[1/3,2]\times[0,1): y=-ix\mod 1\text{ for some }i\leq n\}$.
	
	\begin{figure}[h!]
		\footnotesize
		\begin{center}
		\includegraphics[width=3in]{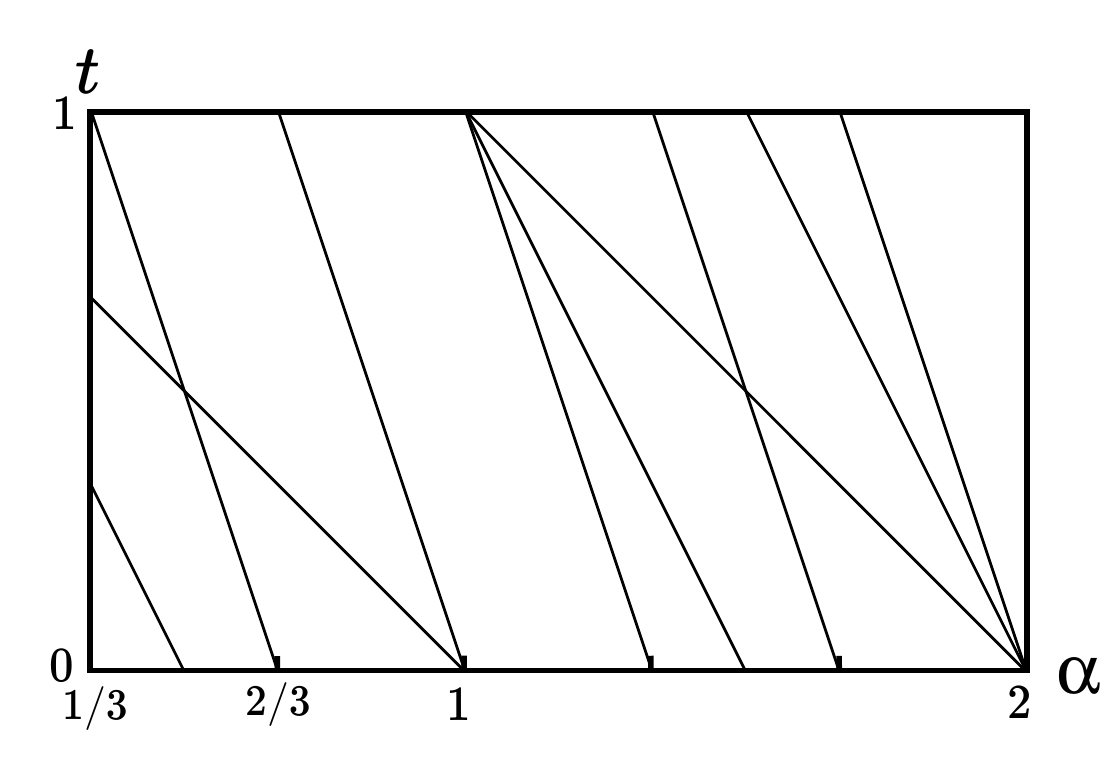}
		\end{center}
		\vspace{-3em}
		\caption{ \footnotesize The partition $\mathcal P_3$.}
	\end{figure}

	\begin{definition}
		For $j\in\N$, define the $\sigma$-algebra 
		$\mathscr B_j = (\id\times\proj_j)^{-1}(\mathscr B)$
		where $\mathscr B$ is the Borel $\sigma$-algebra defined
		on $[1/3,2]\times\R/(6^j\Z)$.
	\end{definition}

	Informally, a
	partition $\mathcal P$ being
	$\mathscr B_j$-measurable means that for a point $(\alpha,t)$,
	$\alpha$ and $t\mod 6^j$ are all you need
	to determine the element of $\mathcal P$ in which it lies.
	Rephrased, $\mathcal P$ gives no
	extra information after the $j$th coordinate of $\mathcal T$.
	Consequently, $\mathcal P_{m,n}$ is $\mathscr B_j$-measurable
	for all $j\geq m$.
	Further, we may interchangeably talk about
	a $\mathscr B_j$-measurable partition of $[1/3,2]\times\T$
	and a Borel-measurable partition of $[1/3,2]\times \R/(6^j\Z)$.
	Given a partition $\mathcal A$ of $[1/3,2]\times\R/(6^j\Z)$, we
	may define a partition $\mathcal P$ of $[1/3,2]\times\T$ by
	the equivalence relation $(\alpha,t)\sim (\alpha',t')$ if 
	$(\alpha, \proj_j(t))$ and $(\alpha', \proj_j(t'))$ lie in
	the same partition element of $\mathcal A$.  In this case,
	we call $\mathcal P$ the \emph{$\mathscr B_j$-measurable extension}
	of $\mathcal A$.

	Let $A=[1/3,2]\times [0,6^i)$ which we will identify
	with $[1/3,2]\times \R/(6^i\Z)$.  Given a finite collection, $L$, of lines in $A$,
	we form a partition of $A$ up to a Lebesgue measure-zero set by 
	taking the connected components of $L^c$.  We call this the
	\emph{geometric partition generated by $L$}.

	Let's consider how $\mathcal P_{m,n}$ and our description of $\mathcal P_n$
	arising from lines relate.  
	
	\begin{definition}
		Let $L_{a,\gamma}^j = \{(x,y)\in[1/3,2]\times \R/(j\Z): y = -ix+\gamma\mod j\text{
		for some }0\leq i\leq a\}$ be the set of lines with slopes in 
		$\{0,-1,-2,\ldots,-a\}$ and offset $\gamma$.
	\end{definition}

	We can now view $\mathcal P_{1,n}$ as being the $\mathscr B_{0}$-measurable extension
	of the partition on $[1/3,2]\times[0,1)$ generated by $L_{n,0}^{1}$.
	Further, the boundary points of $\id\times\proj_j(\mathcal P_{1,n})$
	are precisely the set $\bigcup_{i< 6^j} L_{n,i}^{6^j}$.

	Consider $\hat f^{-1}(\mathcal P_{1,n})$. 
	Since $\hat f^{-1}:[1/3,2]\times \T\to[1/3,2]\times\T$
	either multiplies by $3$ or divides by $2$ (and does so in each coordinate if we view $\hat f$ as
	acting on $(\R^2)^\N$), 
	we see
	\begin{align*}
		\hat f|_{[1/3,2/3]\times \R/(6^j\Z)}^{-1}\big(L_{n,i}^{6^j}\big) &=\\
		&= 
		\big\{(3x,3y\mod 6^j):(x,y)\in L_{n,i}^{6^j}\cap ([1/3,2/3)\times \R/(6^j\Z))\big\} \\
		&=L_{n,3i}^{6^j}\cap ([1,2]\times\R/(6^j\Z))
		\subset L_{n,3i}^{6^j}
	\end{align*}
	and 
	\begin{align*}
		\hat f|_{[2/3,2]\times \R/(6^{j+1}\Z)}^{-1}\big(L_{n,i}^{6^{j+1}}\big) &= \\ 
		&=\big\{(\tfrac{x}{2},\tfrac{y}{2}\mod 6^j):(x,y)\in L_{n,i}^{6^{j+1}}\cap ([2/3,2]\times \R/(6^{j+1}\Z))\big\}\\
		&=L_{n,\frac{i}{2}}^{6^j} \cap ([1/3,1]\times \R/(6^{j+1}\Z))
		\subset L_{n,\frac{i}{2}}^{6^j}.
	\end{align*}

	Illustrated in Figure \ref{FigPartitions} is a truncation of $\proj_3(\hat f^{-i} \mathcal P_{1,1})$.

	\begin{figure}[h!]
	\begin{center}
		\includegraphics[width=1.8in]{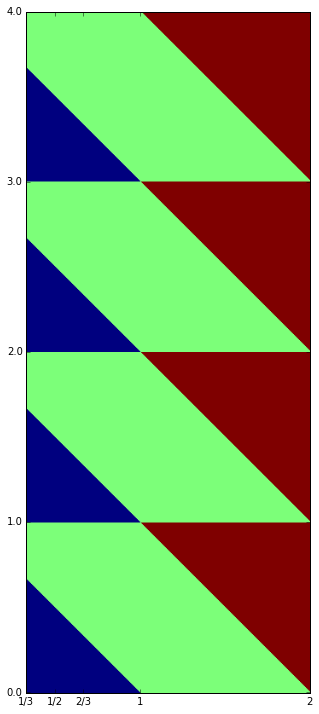}
		\includegraphics[width=1.8in]{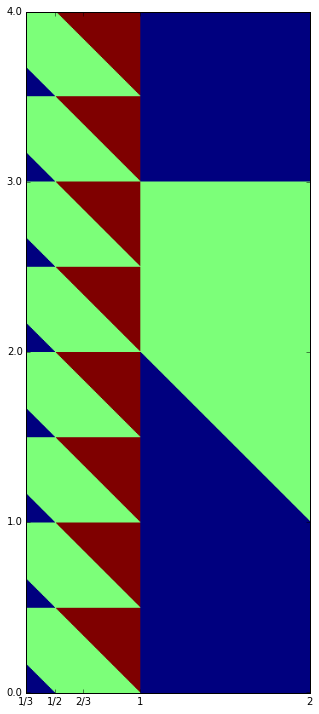}
		\includegraphics[width=1.8in]{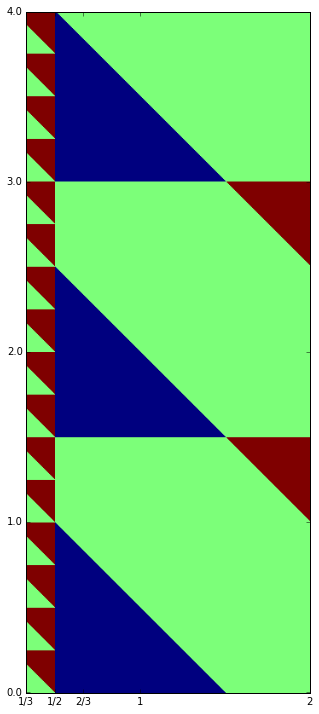}
	\end{center}
	\vspace{-.7cm}
	\caption{
		From left to right, the projection of $\mathcal P_{1,1}$, $\hat f^{-1}\mathcal P_{1,1}$, $\hat f^{-2}\mathcal P_{1,1}$
		onto the third coordinate, truncated to lie in $[1/3,2]\times[0,4)$, and
		colored by whether the symbol at the zero position is $0$, $1$, or $2$.  
	}
	\label{FigPartitions}
	\end{figure}

	\begin{definition}
		Define $\mathrm{rnd}_\alpha:\R\to\Z$ by $\mathrm{rnd}_\alpha(x) = n\in\Z$ 
		whenever $x\in(n-\frac{\log 3}{\log 6}-\frac{\log \alpha}{\log 6}, n+\frac{\log 2}{\log 6}-\frac{\log \alpha}{\log 6})$.
	\end{definition}
	Note that $(n-\frac{\log 3}{\log 6}-\frac{\log \alpha}{\log 6}, n+\frac{\log 2}{\log 6}-\frac{\log \alpha}{\log 6})$
	is an interval of length $1$, so $\mathrm{rnd}_\alpha$ is defined everywhere but the countable set of endpoints. We
	will not attempt to use $\mathrm{rnd}_\alpha$ in any situation where it is not defined.

	\begin{lemma}\label{PropMultiplesof3}
		Suppose $\alpha\notin \Q$.  Then,
		$\hat f^{-j}(\alpha, t) = (\tfrac{3^a}{2^b} \alpha, \tfrac{3^a}{2^b} t)$ and $f^{-j}(\alpha) = \tfrac{3^a}{2^b}\alpha$ where
		\[
			a = j-b\approx \tfrac{\log 2}{\log 6}j\qquad\text{and}\qquad
			b = \mathrm{rnd}_\alpha\left(\tfrac{\log 3}{\log 6}j\right)\approx\tfrac{\log 3}{\log 6}j.
		\]
	\end{lemma}
	The proof of Lemma \ref{PropMultiplesof3} follows directly from solving the system 
	$a+b=j$ and $\tfrac{3^a}{2^b}\alpha\in[1/3,2]$
	with the restriction that $a,b\in \Z$.

	\begin{proposition}\label{PropLunderF}
		Let $B$ be the boundary of $\id\times\proj_m(\mathcal P_{m,n})$ and let
		and $b=\mathrm{rnd}_{f^{-m}(\alpha)}(m\log 3/\log 6)$.  Then $\pi_{\alpha}(B)$ is
		\[
			\pi_{\alpha}\left( \bigcup_{k<2^b6^m} L_{n,\frac{k}{2^b}}^{6^m} \right)
			.
		\]
	\end{proposition}
	\begin{proof}
		Because the boundary of $\mathcal P_{m,n}$ is the boundary of $\bigvee_{i=0}^m f^{-i}(\mathcal P_{1,n})$, 
		we see that $\pi_\alpha(L_{n,\frac{k}{2^b}}^{6^m})$ arises from applying $\hat f^{-m}$ to 
		$\pi_{f^{-m}(\alpha)}(L_{n,k}^{6^{m+b}})$.  This holds for every $k$, which completes the proof.
	\end{proof}

	Iterating $\hat f^{-1}$ and observing how it moves the boundaries of partition
	elements motivates us to define the following refinement of $\mathcal P_{m,n}$.
	
	\begin{definition}
		Let $\mathcal X_{m,n}$ be the $\mathscr B_{m}$-measurable extension
		of the partition generated by $L$ where
		\[
			L = \bigcup_{k<2^m6^m} L_{n,\frac{k}{2^m}}^{6^m}.
		\]
	\end{definition}

	\begin{proposition}
		$\mathcal X_{m,n}$ is a refinement of $\mathcal P_{m,n}$.
	\end{proposition}
	\begin{proof}
		For a fixed $\alpha$, let $b_\alpha$ be the $b$ from Proposition \ref{PropLunderF}, and notice
		that $b_\alpha \leq m$.  It now immediately follows that the set $L$ defining $\mathcal X_{m,n}$ is
		a superset of the set of boundaries of $\id\times\proj_m(\mathcal P_{m,n})$, which completes the proof.
	\end{proof}

	\begin{definition}
		For a partition $\mathcal P$ of $\R$, let $\kappa\mathcal P$ be the \emph{coarseness}
		of $\mathcal P$.  That is,
		\[
			\kappa \mathcal P = \inf_{I\in \mathcal P}\sup\{|b-a|: [a,b]\subset I\}.
		\]
		If $\mathcal P$ is a $\mathscr B_j$-measurable
		partition of $\T$, then $\kappa\mathcal P = \kappa\proj_j(\mathcal P)$.
	\end{definition}

	Given $(\alpha,t)\in [1/3,2]\times\T$, we would like to bound $j$ such that
	$\Orb_{\hat T}^j(\alpha,t) = \{(\alpha,t), \hat T(\alpha, t),\ldots, \allowbreak\hat T^{j-1}(\alpha,t)\}$,
	the $j$-orbit of $(\alpha,t)$ under $\hat T$, intersects every partition element
	of $\pi_\alpha(\mathcal P_{m,n})$.  We can address this in the following way.

	\begin{definition}
		Let $D_\ell^i(\alpha)$ be the smallest $n$ such that
		$\proj_i(\Orb_{\hat T}^n(\alpha,t))$ is $\ell$-dense in $\R/(6^i\Z)$ for any $t$.
	\end{definition}

	Note that the density of $\proj_i(\Orb_{\hat T}^n(\alpha,t))$ is equal to
	the density of $\proj_i(\Orb_{\hat T}^n(\alpha,t'))$,
	and so when computing $D_\ell^i(\alpha)$ we only need to consider a single
	$t$.

	For a fixed $\alpha$, we consider points in $KC$ whose zeroth row
	has rotation number $\alpha$.  For a fixed $t$, consider the $n\times m$ configuration that
	arises based at $(0,0)$ corresponding to $(\alpha,t)$.
	We now see that for any $t'\in \T$, the maximum $\hat T$-waiting time to see
	an occurrence of this $n\times m$ configuration in the point corresponding to
	$(\alpha,t')$ is bounded by
	\[
		D^m_{\kappa\pi_\alpha(\mathcal P_{m,n})}(\alpha) \leq D^m_{\kappa\pi_\alpha(\mathcal X_{m,n})}(\alpha).
	\]

	\begin{proposition}\label{PropXSpacing}
		$\kappa\pi_\alpha(\mathcal P_{m,n})\geq \kappa\pi_\alpha(\mathcal X_{m,n})\geq \min\{|\alpha - \frac{p}{2^mq}|:q\leq n\}$.
	\end{proposition}
	\begin{proof}
		Since $\mathcal X_{m,n}$ is a refinement of $\mathcal P_{m,n}$, $\kappa\pi_\alpha(\mathcal P_{m,n})\geq \kappa\pi_\alpha(\mathcal X_{m,n})$
		follows trivially.

		Let $\hat {\mathcal X}_{m,n}$ be the geometric partition generated by 
		the lines $L = \bigcup_{k<2^m} L_{n,\frac{k}{2^m}}^{1}$. Recalling our
		description of ${\mathcal X}_{m,n}$ in terms of lines, we see that $L$ corresponds
		exactly to the image under $\id\times\proj_0$ of the boundaries of the partition
		elements in ${\mathcal X}_{m,n}$. This shows that
		\[
			\kappa\pi_\alpha(\hat {\mathcal X}_{m,n})=
			\kappa\pi_\alpha(\mathcal X_{m,n}).
		\]
		Thus, we will focus our attention on $\hat {\mathcal X}_{m,n}$.  

		Upon inspecting $L$, we see
		partition elements in $\pi_\alpha(\hat {\mathcal X}_{m,n})$
		have endpoints in the set
		$E=\{\frac{k}{2^m},-\alpha+\frac{k}{2^m},\ldots,-n\alpha+\frac{k}{2^m}\mod 1:
			k<2^m\}$.

		Fix $\alpha$ and observe
		$\kappa\pi_\alpha(\hat {\mathcal X}_{m,n})=d$ is now given by the minimum distance between two
		points in $E$, which is 
		\[
			d=|-i\alpha+\tfrac{k}{2^m} - (-j\alpha+\tfrac{k'}{2^m}) + p|
			=|a\alpha+\tfrac{b}{2^m}+p|,
		\]
		for appropriate $p,a,b\in\Z$.
		Since $\frac{d}{a} = |\alpha + \frac{b+p2^m}{a2^m}|$ for some $p\in \Z$,
		$a\leq q$, and $-k< b< k$, the results follow immediately.
	\end{proof}

	Having obtained a lower bound $\ell$ for $\kappa\pi_\alpha(\mathcal P_{m,n})$, we will 
	now bound above the time it 
	takes an orbit to become $\ell$-dense.

	\begin{definition}
		Define\[
			\mathcal G^{a,b} 
			= \{\alpha: |\alpha-\tfrac{p}{q}|>\tfrac{1}{b}\text{ for } q\leq a\text{ and }p,q\in\N\}.
		\]
	\end{definition}

	Note that $\mathcal G^{a,b}$ could be empty, but a simple estimate shows that $\mathcal G^{a,b}$
	is non-empty if $b>a^2$.

	\begin{proposition}\label{PropGDense}
		$\alpha\in \mathcal G^{ka,b}$ implies $\{0,\alpha, 2\alpha, \ldots, 
		(kb-1)\alpha \mod k\}$ is $\frac{1}{a}$-dense in $\R/(k\Z)$.
	\end{proposition}
	\begin{proof}
		For $x\in \R$, let $\|x\|_k$ represent the distance of $x$ from $k\Z$.
		Suppose $\mathcal G^{ka,b}\neq\emptyset$,
		fix $\alpha\in \mathcal G^{ka,b}$, and let $q\in \N$ be the smallest 
		number  such
		that
		\[
			\|q\alpha\|_k \leq \tfrac{1}{a}.
		\]
		Let $p\in\Z$ be such that $\|q\alpha\|_k = |q\alpha - kp|$.
		By the pigeonhole principle, $q \leq ka$.  By assumption,
		we have 
		$|\alpha - \frac{kp}{q}| > \tfrac{1}{b}$ and so
		\[
			\|q\alpha\|_k = |q\alpha - kp| > \tfrac{q}{b}.
		\]
		
		It would suffice to show that the points 
		\begin{align*}
			X&=\{0, \|q\alpha\|_k,
				2\|q\alpha\|_k, \ldots, (\ceil{b/q}k-1)\|q\alpha\|_k\}  \\
			&= 
				\{0,q\alpha, 2q\alpha, \ldots, (\ceil{b/q}k-1)q\alpha\mod k\}
		\end{align*}
		are
		$\frac{1}{a}$-dense since they form a subset of 
		$\{0,\alpha, 2\alpha, \ldots, 
		(kb-1)\alpha \mod k\}$.  Since consecutive points in $X$ are separated by 
		a distance of less that $\frac{1}{a}$, we need only show that the last point
		satisfies $(\ceil{b/q}k-1)\|q\alpha\|_k\geq k-1/a$.  But this is implied by
		the fact that $\frac{b}{q}\|q\alpha\|_k > 1$, which completes the proof.
	\end{proof}

	\begin{proposition}\label{PropGSparse}
		$\alpha\in \mathcal G^{ka,b}$ implies $\{0,\alpha, 2\alpha, \ldots, 
		ka\alpha \mod k\}$ is $\frac{1}{b}$-sparse in $\R/(k\Z)$.  That is,
		no two points are within $1/b$ of each other.
	\end{proposition}
	\begin{proof}
		Suppose $\mathcal G^{ka,b}\neq\emptyset$.
		Fix $\alpha\in \mathcal G^{ka,b}$ and note that to prove 
		$\frac{1}{b}$-sparsity of
		$\{0,\alpha, 2\alpha, \ldots, ka\alpha \mod k\}$ we only need
		to show $\|r\alpha\|_k > \frac{1}{b}$ for all $0<r\leq ka$.
		
		Choose $p,q$ to minimize 
		$|q\alpha-kp|$ subject to $0<q\leq ka$.
		We then have that $\|r\alpha\|_k$ is minimized by
		\[
			\|q\alpha\|_k = |q\alpha-kp| \geq |\alpha-\tfrac{kp}{q}| > \tfrac{1}{b},
		\]
		with the last inequality following by assumption.
	\end{proof}

	The previous propositions show a symmetry in $\mathcal G^{a,b}$.
	Namely, if $\alpha\in \mathcal G^{a,b}$, then the $b$-orbit of rotation
	by $\alpha$ is $1/a$-dense and the $(a+1)$-orbit of rotation by $\alpha$ is
	$1/b$-sparse.

	\begin{corollary}\label{PropBiggerThanOneOverB}
		If $\alpha\in \mathcal G^{2^m n,b}$ then $\kappa\pi_{\alpha}(\mathcal X_{m,n}) > 
		\frac{1}{b}$.
	\end{corollary}
	\begin{proof}
		Fix $\alpha\in \mathcal G^{2^mn,b}$.
		By Proposition \ref{PropXSpacing}, 
		$\kappa \pi_\alpha(\mathcal X_{m,n})\geq 
		\min\{|\alpha - \frac{p}{2^mq}|:q\leq n\text{ and }p,q\in\N\}$.
		By the assumption that $\alpha\in \mathcal G^{2^mn,b}$, we have
		$|\alpha - \frac{p}{2^mq}|> \frac{1}{b}$.
	\end{proof}

	\begin{proposition}\label{PropWaitingTime}
		If $\alpha\in\mathcal G^{2^mn,b}\cap \mathcal G^{6^mb,c}$ then
		\[
			D^m_{\kappa\pi_\alpha(\mathcal X_{m,n})}(\alpha) \leq 6^mc.
		\]
	\end{proposition}
	\begin{proof}
		Fix $\alpha\in\mathcal G^{2^mn,b}\cap \mathcal G^{6^mb,c}$.
		Since $\alpha\in \mathcal G^{2^mn,b}$, Corollary \ref{PropBiggerThanOneOverB}
		implies $\kappa\pi_\alpha(\mathcal X_{m,n})>\frac{1}{b}$
		and so $\kappa\pi_\alpha(\proj_m(\mathcal X_{m,n}))>\frac{1}{b}$.
		By Proposition \ref{PropGDense} applied to $\mathcal G^{6^mb,c}$, we have that 
		$E=\{0,\alpha, \ldots, (6^mc-1)\alpha\mod 6^m\}$ is $\frac{1}{b}$-dense
		in $[0,6^m]$, and so $E$ intersects every partition
		element of $\pi_\alpha(\mathcal X_{m,n})$, which completes the proof.
	\end{proof}

	We have identified $\alpha$'s that give us good return times, but
	$\mathcal G^{2^mn,b}\cap\mathcal G^{6^mb,c}$ could be empty.
	Next we will find constraints on $b,c$ to avoid this and guarantee
	us
	some useful properties.

	\begin{definition}
		Given a set $X$ and a collection of sets $\mathcal C$, we say $(X,\mathcal C)$
		has the \emph{intersection property} if for all $I\in \mathcal C$,
		$X\cap I\neq \emptyset$. 
		If $X\subset \R$, we say $X$ is \emph{$\delta$-fat} relative to $\mathcal C$
		if for all $I\in \mathcal C$, $X\cap I$ contains an interval of width $\delta$.
	\end{definition}
	\begin{definition}
		Let $\mathcal F_n$ be the partition of $\R$ whose elements are
		of the form $[a,b)$ where $a,b$ are consecutive points in $\{\frac{p}{q}:q\leq n\}$.
		That is $\mathcal F_n$ is the partition of $\R$ into half-open intervals whose endpoints
		are consecutive Farey fractions with denominator bounded by $n$.
	\end{definition}

	\begin{proposition}\label{PropFareyGoodEnough}
		Let $p:[1/3,2]\times\T\to [1/3,2]$ be projection onto the first coordinate.
		Let $X\subset \R$. If $(X,\mathcal F_{2^m n})$ has the intersection property, then for any element
		$E\in \mathcal X_{m,n}$, $X\cap p(E)\neq \emptyset$.
	\end{proposition}
	\begin{proof}
		Let $\hat {\mathcal X}_{m,n}=\id\times\proj_0(\mathcal X_{m,n})$ 
		and note it is sufficient
		to show that if $(X,\mathcal F_{2^m n})$ has the intersection
		property, then for any element
		$E\in\hat {\mathcal X}_{m,n}$, we have $X\cap p(E)\neq\emptyset$.

		Recalling the description of $\hat {\mathcal X}_{m,n}$ 
		in terms of lines, we see
		that $\hat {\mathcal X}_{m,n}$ consists of 
		polygonal regions whose corners have coordinates of the form 
		$\frac{p}{2^mq}$ for some $q\leq n$.  
		Since every element of $\hat {\mathcal X}_{m,n}$ 
		contains an open set, we see that for all $P\in \hat {\mathcal X}_{m,n}$, 
		there exists
		$I\in \mathcal F_{2^mn}$ so that $I\subset p(P)$ (possibly ignoring
		some points along the boundary of $P$), which completes the proof.
	\end{proof}

	\begin{proposition}\label{PropIntersectProp}
		If $b \geq 4a^2$, $c^2\geq 4b$, and $d\geq 4c^2$ then 
		$\mathcal G^{a,b}\cap \mathcal G^{c,d}$ is $\frac{2}{d}$-fat
		relative to $\mathcal F_a$.
	\end{proposition}
	\begin{proof}
		By definition $\mathcal G^{x,y}$ is constructed by removing balls
		of radius $1/y$ centered at points $\frac{p}{q}$ with $q\leq x$.
		If $q,q'\leq x$, then $|\frac{p}{q}-\frac{p'}{q'}|>\frac{1}{x^2}$.
		Thus, if $y> 4x^2$, not only will $\mathcal G^{x,y}$ intersect
		every element of $\mathcal F_x$, but it will do so with diameter at least
		\[
			\frac{1}{x^2}-\frac{2}{y} = \frac{1}{2x^2}.
		\]

		Suppose $a,b,c,d$ satisfy $b \geq 4a^2$, $c^2\geq 4b$, and $d\geq 4c^2$.
		Every gap in $\mathcal G^{c,d}$ is of size
		$\frac{2}{d}<\frac{1}{2c^2}$ and every interval in $\mathcal G^{c,d}$
		has size at least $\frac{1}{2c^2}$.  Thus, the intersection of 
		$\mathcal G^{c,d}$ with an interval of width $\frac{1}{2b}$ must contain an
		interval of width at least
		\[
			\min\left\{\frac{1}{2c^2},\frac{1}{2b} - \frac{2}{2c^2}\right\}
			\geq 
			\min\left\{\frac{1}{2c^2},\frac{2}{c^2} - \frac{1}{c^2}\right\}
			=\frac{1}{2c^2}\geq \frac{2}{d}.
		\]
		Noticing that the smallest interval in $\mathcal G^{a,b}$
		is of size at least $\frac{2}{b}>\frac{1}{2b}$ completes the proof.
	\end{proof}

	We can now identify a set of $\alpha$'s that have good waiting times.

	\begin{definition}
		Let $\mathcal W_{n\times m} = \mathcal G^{a,b}\cap \mathcal G^{c,d}$ where $a=2^mn$, 
		$b=2^{2m+2}n^2$, $c=6^{m}2^{2m+2}n^2$, and $d=6^{4m+3}n^4$.
	\end{definition}

	Notice that the parameters $a,b,c,d$ in $\mathcal W_{n\times m}$
	were carefully chosen to satisfy the conditions of Proposition \ref{PropIntersectProp}
	and Proposition \ref{PropWaitingTime}.

	\begin{theorem}\label{PropHorizWaitingTime}
		Let $c$ be an $n\times m$ configuration in $KC$ and $A=\{0,\ldots, m-1\}
		\times \{0,\ldots, n-1\}$.  Then there exists
		an interval $I_c\subset  \mathcal W_{n\times m}$ of width $2/(6^{4m+3}n^4)$
		so that for every $\alpha\in I_c$ and every $t\in \T$,
		there exists a $j<6^{5m+3}n^4$ so that
		\[
			K\circ \hat T^j (\alpha, t)|_A = c.
		\]
	\end{theorem}
	\begin{proof}
		Given the framework we have established, the proof is straightforward.

		Proposition \ref{PropIntersectProp} tells us that $\mathcal W_{n\times m}$ is
		$2/(6^{4m+3}n^4)$-fat relative to $\mathcal F_{2^mn}$,
		and so by Proposition
		\ref{PropFareyGoodEnough}, we have that there exists
		an interval $I_c\subset  \mathcal W_{n\times m}$ of width $2/(6^{4m+3}n^4)$
		so that for every $\alpha\in I_c$ there exists $t\in \T$ so
		$K(\alpha, t)|_A = c$.
		
		Fix $I_c$ and $\alpha\in I_c$.
		By Proposition \ref{PropWaitingTime},
		\[
			D_{\kappa\pi_\alpha(\mathcal X_{m,n})}^m(\alpha) \leq 6^m6^{4m+3}n^4 = 6^{5m+3}n^4,
		\]
		and so we will see $c$ in less than $6^{5m+3}n^4$ applications of $\hat T$,
		which completes the proof.
	\end{proof}

	Theorem \ref{PropHorizWaitingTime} gives the bulk of the proof of
	Theorem \ref{ThmKCBounds}.
	If we have an $n\times m$ configuration $c$ in mind, we know there is
	an open interval $I_c$ of angle parameters where we will see
	$c$ in a horizontal orbit of no more than $6^{5m+3}n^4$ steps.
	Since orbits under $\hat S$ are dense in the first coordinate, we
	know that if we bound how long it takes for an $\hat S$-orbit
	(equivalently an $f$-orbit)
	to become $|I_c|$-dense, we have a bound on the minimum size
	of a rectangle that contains the configuration $c$.
\subsection{Asymptotic Density of Orbits Under $f$}
	\begin{definition}[Irrationality Measure]
		For a number $\alpha\in\R$, the \emph{irrationality measure}
		of $\alpha$ is
		\[
			\eta(\alpha) = 
			\inf\left\{\gamma: \left|\alpha-\frac{p}{q}\right|<\frac{1}{q^\gamma}
			\text{ for only finitely many }p,q\in\Z\right\}.
		\]
	\end{definition}

	\begin{proposition}[Rhin \cite{rhin}]\label{PropRhin}
		For $u_0,u_1,u_2\in\Z$ and $H=\max\{|u_1|,|u_2|\}$, we have
		that if $H$ is sufficiently large,
		\[
			|u_0+u_1\log 2 + u_2\log 3| \geq H^{-7.616},
		\]
		and if $H\geq 2$, we have the universal bound
		\[
			|u_0+u_1\log 2 + u_2\log 3| \geq H^{-13.3}.
		\]

	\end{proposition}

	\begin{corollary}\label{PropIrrationalityMeasure}
		$\eta(\log 2/\log 6) \leq 8.616$ and $\left|\frac{\log 2}{\log 6}-\frac{p}{q}\right|\geq \frac{1/\log 6}{q^{14.3}}$
		if $q\geq 2$.
	\end{corollary}
	\begin{proof}
		Let $x=\left|\frac{\log 2}{\log 6}-\frac{p}{q}\right|$.  By algebraic manipulation,
		we deduce
		\[
			xq\log 6 = |(q-p)\log 2 - p\log 3|.
		\]
		And so by Proposition \ref{PropRhin} and the fact that 
		$\max\{|q-p|,|p|\} \leq q$, we have that asymptotically, $xq\log 6 \geq q^{-7.616}$, which produces a bound
		of 
		$
			x \geq \frac{1/\log 6}{q^{8.616}}.
		$
		Alternatively, we may use the bound $xq\log 6 \geq q^{-13.3}$, which holds for all $q\geq 2$.
	\end{proof}

	\begin{proposition}\label{PropfWaitingTime}
		Fix $\delta > 0$ and
		let
		$k_\ell \geq (\frac{3}{\ell\log 6})^{8.616 + \delta}$.
		Then, for sufficiently small $\ell$, the $k_\ell$-orbit of any $x\in [1/3,2]$ under $f$ is $\ell$-dense.  That is
		\[
			\{x,f(x),f^2(x),\ldots f^{k_\ell-1}(x)\}
		\]
		is $\ell$-dense for any $x\in[1/3,2]$.  Further, if $k_\ell \geq (\frac{3}{\ell\log 6})^{14.3}\log 6$ and $1/\ell\geq 2$, then
		the $k_\ell$ orbit of any point $x\in[1/3,2]$ under $f$ is $\ell$-dense.
	\end{proposition}
	\begin{proof}
		Let $\phi$ be the conjugacy from Proposition \ref{PropfAperiodic} 
		between $f$ and rotation by $\frac{\log 2}{\log 6}$.
		We have that $|\phi'|$ attains a maximum value of $\frac{3}{\log 6}$.  Thus, 
		to ensure an orbit segment under $f$ is $\ell$-dense, we must have that
		the image of an orbit segment
		under $\phi\circ f\circ \phi^{-1}=R_{\frac{\log 2}{\log 6}}$ 
		is $\frac{\ell\log 6}{3}$-dense.
		Let $\eta=\eta(\log 2/\log 6)$ be the irrationality measure of $\log2/\log 6$.
		Fix $\delta>0$.  We then have, by the definition
		of the irrationality measure, $\log2/\log 6\in \mathcal G^{k, k^{\eta+\delta}}$
		for all sufficiently large $k$.
		Applying Proposition \ref{PropGDense} and using 
		Corollary \ref{PropIrrationalityMeasure} to bound $\eta$ now completes the proof of the first claim.

		For the second claim, note that Corollary \ref{PropIrrationalityMeasure} implies that $\frac{\log 2}{\log 6}
		\in \mathcal G^{k,k^{14.3}\log 6}$ for any $k\geq 2$.  The proof then follows similarly.
	\end{proof}

	\begin{theorem*}[Theorem \ref{ThmKCBounds}]
		Let $\eta=\eta(\log 2/\log 6)$.  Every legal $n\times m$ configuration
		in $KC$ occurs in every $B\times A$ configuration in $KC$ where
		\[
			A = \left(\frac{324}{\log 6}6^{4m}n^{4}\right)^\eta
			< 6^{34.464m+25}n^{34.464}\qquad\text{and}\qquad
			B = 6^{5m+3}n^4
		\]
		for sufficiently large $m+n$.

		Further, for all $m,n$ we have that a copy of every legal $n\times m$ configuration
		in $KC$ occurs in every $B\times A$ configuration in $KC$ where
		\[
			A = \left(\frac{324}{\log 6}6^{4m}n^{4}\right)^{14.3} \log 6
			\qquad\text{and}\qquad
			B = 6^{5m+3}n^4
		\]
	\end{theorem*}
	\begin{proof}
		Let $C=\{0,\ldots,n-1\}\times \{0,\ldots, m-1\}$
		and let $c$ be a legal $n\times m$ configuration.
		Fix $I_c\subset \mathcal W_{n\times m}$ as in 
		Theorem \ref{PropHorizWaitingTime}.
		We now have that for any $(\alpha,t)\in I_c\times \T$,
		$K\circ \hat T^j(\alpha, t)|_C=c$ for some $j<6^{5m+3}n^4$.

		Since $I_c$ is of length at least $2/(6^{4m+3}n^4)$, by
		Proposition \ref{PropfWaitingTime} with $\ell = 2/(6^{4m+3}n^4)$,
		we see that for any $(\alpha, t)\in[1/3,2]\times \T$, we have
		$\hat S^j(\alpha, t)\in I_c\times\T$ for some $j<(3\cdot 6^{4m+3}n^4/(2\log 6))^\eta$.

		We now have a bound on how many applications of $\hat T$ and $\hat S$ it takes to 
		land in a particular element of $\mathcal P_{n,m}$, which gives bounds on $A$ and $B$.

		Alternatively, using the second part of proposition \ref{PropfWaitingTime} 
		with $\ell = 2/(6^{4m+3}n^4)$
		we get a bound for all $m\geq 2$.
	\end{proof}

  \begin{acknowledgements}
    I would like to thank the referee for a careful reading and
    useful suggestions.
  \end{acknowledgements}

\footnotesize
%
%
%
%
%

\bibliographystyle{abbrv}
\bibliography{JasonSiefkenKCMinimality}

\begin{thebibliography}{1}

\bibitem{berenstein}
C.~A. Berenstein and D.~Lavine.
\newblock On the number of digital straight line segments.
\newblock {\em Pattern Analysis and Machine Intelligence, IEEE Transactions
  on}, 10(6):880--887, Nov 1988.

\bibitem{durand}
B.~Durand, G.~Gamard, and A.~Grandjean.
\newblock Aperiodic tilings and entropy.
\newblock In {\em 18th International Conference on Developments in Language
  Theory, DLT 2014, Ekaterinburg, Russia}, 2014.

\bibitem{eigen}
S.~Eigen, J.~Navarro, and V.~S. Prasad.
\newblock An aperiodic tiling using a dynamical system and {B}eatty sequences.
\newblock In {\em Dynamics, ergodic theory, and geometry}, volume~54 of {\em
  Math. Sci. Res. Inst. Publ.}, pages 223--241. Cambridge Univ. Press,
  Cambridge, 2007.

\bibitem{fogg}
N.~P. Fogg.
\newblock {\em Substitutions in dynamics, arithmetics and combinatorics},
  volume 1794 of {\em Lecture Notes in Mathematics}.
\newblock Springer-Verlag, Berlin, 2002.
\newblock Edited by V. Berth{\'e}, S. Ferenczi, C. Mauduit and A. Siegel.

\bibitem{liousse}
I.~Liousse.
\newblock P{L} homeomorphisms of the circle which are piecewise {$C^1$}
  conjugate to irrational rotations.
\newblock {\em Bull. Braz. Math. Soc. (N.S.)}, 35(2):269--280, 2004.

\bibitem{lothaire}
M.~Lothaire.
\newblock {\em Algebraic combinatorics on words}, volume~90 of {\em
  Encyclopedia of Mathematics and its Applications}.
\newblock Cambridge University Press, Cambridge, 2002.
\newblock A collective work by Jean Berstel, Dominique Perrin, Patrice Seebold,
  Julien Cassaigne, Aldo De Luca, Steffano Varricchio, Alain Lascoux, Bernard
  Leclerc, Jean-Yves Thibon, Veronique Bruyere, Christiane Frougny, Filippo
  Mignosi, Antonio Restivo, Christophe Reutenauer, Dominique Foata, Guo-Niu
  Han, Jacques Desarmenien, Volker Diekert, Tero Harju, Juhani Karhumaki and
  Wojciech Plandowski, With a preface by Berstel and Perrin.

\bibitem{rhin}
G.~Rhin.
\newblock Approximants de {P}ad\'e et mesures effectives d'irrationalit\'e.
\newblock In {\em S\'eminaire de {T}h\'eorie des {N}ombres, {P}aris 1985--86},
  volume~71 of {\em Progr. Math.}, pages 155--164. Birkh\"auser Boston, Boston,
  MA, 1987.

\bibitem{robbie}
E.~A. Robinson, Jr.
\newblock The tilings of {K}ari and {C}ulik.
\newblock Presented at Numeration: Mathematics and Computer Science, CIRM,
  Marseilles, 2009.

\end{thebibliography}

\end{document}